\documentclass[12pt,a4paper,reqno]{amsart}
\usepackage[T1]{fontenc}
\usepackage{amsmath,amssymb,ifthen}
\usepackage{mathrsfs}
\usepackage{fullpage,verbatim}
\usepackage{url}
\usepackage{graphicx,psfrag,subfigure}
\usepackage[dvips]{color}

\def\H{\widetilde{H}}

\def\R{{\mathbb R}}
\def\N{{\mathbb N}}

\def\ba{\boldsymbol{a}}
\def\bb{\boldsymbol{b}}
\def\bc{\boldsymbol{c}}
\def\bv{\boldsymbol{v}}
\def\bu{\boldsymbol{u}}
\def\bE{\boldsymbol{E}}

\def\bw{\boldsymbol{w}}
\def\bm{\boldsymbol{m}}
\def\bxi{\boldsymbol{\xi}}
\def\bH{\boldsymbol{H}}
\def\bphi{\boldsymbol{\phi}}

\def\bpsi{\boldsymbol{\psi}}

\def\bn{\boldsymbol{n}}
\def\btau{\boldsymbol{\tau}}

\def\NN{{\mathcal N}}
\def\DD{{\mathcal D}}
\def\KK{{\mathcal K}}

\def\PP{{\mathcal P}}

\def\SS{{\mathcal S}}
\def\TT{{\mathcal T}}

\def\XX{{\mathcal X}}

\def\Ltwo#1{{\mathbb L}^{2}{(#1})}

\def\L{{\mathbb{L}}}
\def\H{{\mathbb{H}}}

\def\Hone#1{{\mathbb H}^{1}{(#1})}
\def\Hcurl#1{{\mathbb H}({\rm curl},#1)}

\def\dive{{\rm div}}

\def\norm#1#2{\|#1\|_{#2}}

\def\set#1#2{\big\{#1\,:\,#2\big\}}

\def\eps{\varepsilon}

\def\slp{\mathfrak{V}} % simple-layer potential
\def\dlp{\mathfrak{K}} % double-layer potential
\def\hyp{\mathfrak{W}} % hypersingular integral operator
\def\dtn{\mathfrak{S}} % Dirichlet-to-Neumann map
\def\mfrac#1#2{\mbox{$\frac{#1}{#2}$}}

\def\normL2#1#2{\|#1\|_{L^2(#2)}}

\newcommand{\dual}[3][]{#1\langle#2\,,\,#3#1\rangle}

\newtheorem{theorem}{Theorem}

\newtheorem{lemma}[theorem]{Lemma}

\newtheorem{algorithm}[theorem]{Algorithm}
\newtheorem{definition}[theorem]{Definition}
\newtheorem{remark}[theorem]{Remark}

\numberwithin{equation}{section}
\begin{document}%%%%%%%%%%%%%%%%%%%%%%%%%%%%%%%%%%%%%%%%%%%%%%%%%%%%%
\title{
The Eddy Current--LLG Equations--Part I: FEM-BEM Coupling
}
\author{Michael Feischl}
\address{School of Mathematics and Statistics,
         The University of New South Wales,
         Sydney 2052, Australia}
\email{m.feischl@unsw.edu.au}

\author{Thanh Tran}
\address{School of Mathematics and Statistics,
         The University of New South Wales,
         Sydney 2052, Australia}
\email{thanh.tran@unsw.edu.au}
\thanks{Supported by the Australian Research Council under
grant numbers DP120101886 and DP160101755}


\date{\today}

\begin{abstract}
We analyse a numerical method for the coupled
system of the eddy current equations in $\R^3$ with the
Landau-Lifshitz-Gilbert equation in a bounded domain.
The unbounded domain is discretised by means of
finite-element/boundary-element coupling. Even though the
considered problem is strongly nonlinear, the numerical approach is
constructed such that only two linear systems per time step have to
be solved. In this first part of the paper,
we prove unconditional weak convergence (of a
subsequence) of the finite-element solutions
towards a weak solution. A priori error estimates 
%and numerical experiments 
will be presented in the second part.
\end{abstract}
\maketitle
\section{Introduction}
This paper deals with the coupling of finite element and boundary
element methods to solve the system of the eddy current
equations in
the whole 3D spatial space and the Landau-Lifshitz-Gilbert equation
(LLG), the so-called ELLG system or equations. The system is also called the
quasi-static Maxwell-LLG (MLLG) system.

The LLG is widely considered as a valid model of micromagnetic
phenomena occurring in, e.g., magnetic sensors, recording heads,
and magneto-resistive storage device~\cite{Gil,LL,prohl2001}. 
Classical results concerning existence and non-uniqueness of solutions can
be found in~\cite{as,vis}. In a ferro-magnetic material, magnetisation is
created or affected by external electro-magnetic fields. It is therefore
necessary to augment the Maxwell system with the LLG, which describes the
influence of ferromagnet; see e.g.~\cite{cim, KruzikProhl06, vis}.
Existence, regularity and local uniqueness for the MLLG equations are studied
in~\cite{CimExist}.

Throughout the literature, there are various works on numerical
approximation methods for the LLG, ELLG, and MLLG 
equations~\cite{alouges,alok,bako,bapr,cim,ellg,thanh} (the list is not
exhausted),
and even with the full Maxwell system on bounded
domains~\cite{BBP,MLLG}, and in the whole~$\R^3$~\cite{CarFab98}.
Originating from the seminal
work~\cite{alouges}, the recent works~\cite{ellg,thanh} consider a
similar numeric integrator for a bounded domain.

This work studies the ELLG equations where we 
consider the electromagnetic field on the whole $\R^3$ and do not need to
introduce artificial boundaries. 
Differently from~\cite{CarFab98} where the Faedo-Galerkin method is
used to prove existence of weak solutions,
we extend the analysis for the integrator used in~\cite{alouges, ellg,
thanh} to a
finite-element/boundary-element (FEM/BEM) discretisation of the
eddy current part on $\R^3$. This is inspired by the FEM/BEM coupling
approach designed for the pure eddy current problem
in~\cite{dual}, which
allows to treat unbounded domains without introducing artificial
boundaries. Two approaches are proposed in~\cite{dual}: the
so-called ``magnetic (or $\bH$-based)
approach'' which eliminates the electric
field, retaining only the magnetic field as the unknown in the
system, and the ``electric (or $\bE$-based)
approach'' which considers a
primitive of the electric field as the only unknown. The
coupling of the eddy-current system with the LLG dictates that
the first approach is more appropriate; see~\eqref{eq:strong}.

The main result of this first part is the weak convergence of the 
discrete approximation towards a weak solution without any
condition on the space and time discretisation. This also proves
the existence of weak solutions.

The remainder of this part is organised as follows.
Section~\ref{section:model} introduces the coupled problem and the
notation, presents the numerical algorithm, and states the main
result of this part of
the paper. Section~\ref{sec:pro} is devoted to the proof of
this main result. Numerical results are presented in
Section~\ref{section:numerics}. The second part of this paper~\cite{FeiTraII} proves a~priori estimates for the proposed algorithm.
%Section~\ref{section:numerics} provides a numerical example underlining
%the theoretical result. The final section, the Appendix, contains the
%proofs of some rather elementary or well-known results.

 \section{Model Problem \& Main Result}\label{section:model}
 \subsection{The problem}\label{subsec:pro}
 Consider a bounded Lipschitz domain $D\subset \R^3$ with connected
boundary $\Gamma$ having the outward normal vector $\bn$.
We define $D^\ast:=\R^3\setminus\overline D$,
$D_T:=(0,T)\times D$, $\Gamma_T := (0,T)\times\Gamma$, 
$D_T^{\ast}:=(0,T)\times D^\ast$,
and $\R^3_T := (0,T)\times\R^3$ for $T>0$.
 We start with the quasi-static approximation of the full Maxwell-LLG system from~\cite{vis} which reads as
\begin{subequations}\label{eq:strong}
 \begin{alignat}{2}
  \bm_t - \alpha\bm\times\bm_t &= -\bm \times \bH_{\rm eff}
 &&\quad\text{in }D_T,\label{eq:llg}\\
\sigma\bE -\nabla\times\bH&=0&&\quad\text{in }\R^3_T,\label{eq:MLLG1}\\
\mu_0\bH_t +\nabla \times\bE &=-\mu_0\widetilde\bm_t&&\quad\text{in }\R^3_T,\label{eq:MLLG2}\\
{\rm div}(\bH+\widetilde\bm)&=0 &&\quad\text{in }\R^3_T,\label{eq:MLLG3}\\
{\rm div}(\bE)&=0&&\quad\text{in }D^\ast_T,
\end{alignat}
\end{subequations}
where $\widetilde \bm$ is
the zero extension of $\bm$ to $\R^3$ 
and $\bH_{\rm eff}$ is the
effective field defined by $\bH_{\rm eff}= C_e\Delta\bm+\bH$ for
some constant $C_e>0$.
Here the parameter $\alpha>0$ and permability $\mu_0\geq0$ are
constants, whereas the conductivity $\sigma$ takes a constant
positive value in
$D$ and the zero value in $D^\ast$.
Equation~\eqref{eq:MLLG3} is understood in the
distributional sense because there is a jump of~$\widetilde\bm$
across~$\Gamma$.

It follows from~\eqref{eq:llg} that $|\bm|$ is constant. We
follow the usual practice to normalise~$|\bm|$ (and thus
the same condition is required for $|\bm^0|$).
The following conditions are imposed on the solutions
of~\eqref{eq:strong}:
\begin{subequations}\label{eq:con}
\begin{alignat}{2}
\partial_n\bm&=0
&& \quad\text{on }\Gamma_T,\label{eq:con1} 
\\
|\bm| &=1
&& \quad\text{in } D_T, \label{eq:con2}
\\
\bm(0,\cdot) &= \bm^0
&& \quad\text{in } D, \label{eq:con3}
\\
\bH(0,\cdot) &= \bH^0
&& \quad\text{in } \R^3,
\\
\bE(0,\cdot) &= \bE^0
&& \quad\text{in } \R^3,
\\
|\bH(t,x)|&=\mathcal{O}(|x|^{-1})
&& \quad\text{as }|x|\to \infty,
%\\
%|\nabla\times\bH(t,x)|&=\mathcal{O}(|x|^{-1})
%&& \quad\text{as }|x|\to \infty,
\end{alignat}
\end{subequations}
where $\partial_n$ denotes the normal derivative.
The initial data~$\bm^0$ and~$\bH^0$ satisfy~$|\bm^0|=1$ in~$D$ and
\begin{align}\label{eq:ini}
\begin{split}
 {\rm div}(\bH^0 + \widetilde \bm^0)&=0\quad\text{in }\R^3.
%\\
%(\bH^0 + \bm^0)\cdot \bn&=0\quad\text{on }\Gamma.
\end{split}
\end{align}

%The condition~\eqref{eq:con2} together with basic
%properties of the cross product leads to the following equivalent
%formulation of~\eqref{eq:llg}:
%\begin{align}\label{eq:llg2}
%\alpha\bm_t+\bm\times\bm_t= \bH_{\rm eff}-(\bm\cdot \bH_{\rm eff})\bm
%\quad\text{in } D_T.
%\end{align}

Below, we focus on an $\bH$-based formulation of the problem. 
%However, %as stated in~\cite[Equation~(15)]{dual}, 
It is possible to recover $\bE$ once $\bH$ and $\bm$ are known;
see~\eqref{eq:el}

\subsection{Function spaces and notations}\label{subsec:fun spa}
Before introducing the concept of weak solutions
to problem~\eqref{eq:strong}--\eqref{eq:con} 
we need the following definitions of function
spaces. Let $\Ltwo{D}:=L^2(D;\R^3)$ and $\Hcurl{D}:=\set{\bw\in
\Ltwo{D}}{\nabla\times\bw\in\Ltwo{D}}$.
We define $H^{1/2}(\Gamma)$ as the usual trace space of $H^1(D)$ and define its dual space $H^{-1/2}(\Gamma)$ by extending
the $L^2$-inner product on $\Gamma$.
For convenience we denote
\[
\XX:=\set{(\bxi,\zeta)\in\Hcurl{D}\times
H^{1/2}(\Gamma)}{\bn\times\bxi|_\Gamma =\bn\times \nabla_\Gamma\zeta
\text{ in the sense of traces}}.
\]
Recall that~$\bn\times\bxi|_\Gamma$ is the tangential trace (or
twisted tangential trace) of~$\bxi$, and $\nabla_\Gamma\zeta$ is the
surface gradient of~$\zeta$. Their definitions and properties can
be found in~\cite{buffa, buffa2}.
%(Note that this is well-defined since $\nabla_\Gamma\colon H^{1/2}(\Gamma)\to H^{-1/2}(\Gamma)$ and $\nabla\times(\cdot)\colon\Hcurl{D}\to
%H^{-1/2}(\Gamma)$. Moreover, $\norm{\bn\times \nabla_\Gamma\lambda}{H^{-1/2}(\Gamma)}\lesssim \norm{ \nabla_\Gamma\lambda}{H^{-1/2}(\Gamma)}$ at least on polygonal surfaces.) 

Finally, if $X$ is a normed vector space then $L^2(0,T;X)$,
$H^m(0,T;X)$, and $W^{m,p}(0,T;X)$ 
denote the usual corresponding Lebesgues and Sobolev spaces of functions
defined on $(0,T)$ and taking values in $X$.

We finish this subsection with the clarification of the meaning of the
cross product between different mathematical objects.
For any vector functions $\bu, \bv, \bw$ %\in \Hone{D}$
we denote
\begin{gather*}
\bu\times\nabla\bv
%&
:=
\left(
\bu\times\frac{\partial\bv}{\partial x_1},
\bu\times\frac{\partial\bv}{\partial x_2},
\bu\times\frac{\partial\bv}{\partial x_3}
\right), 
%\\
\quad
\nabla\bu\times\nabla\bv
%&
:=
\sum_{i=1}^3
\frac{\partial\bu}{\partial x_i}
\times
\frac{\partial\bv}{\partial x_i} 
\\
\intertext{and}
%\dual{\bu\times\nabla\bv}{\nabla\bw}_{\Ltwo{D}}
(\bu\times\nabla\bv)\cdot \nabla\bw
%&
:=
\sum_{i=1}^3
\left(
\bu\times\frac{\partial\bv}{\partial x_i}
\right)
\cdot \frac{\partial\bw}{\partial x_i}.
%\dual{\bu\times\frac{\partial\bv}{\partial x_i}}%
%{\frac{\partial\bw}{\partial x_i}}_{\Ltwo{D}}.
\end{gather*}

\subsection{Weak solutions}\label{subsec:wea sol}
A weak formulation for~\eqref{eq:llg} is well-known, see
e.g.~\cite{alouges, thanh}. 
Indeed, by multiplying~\eqref{eq:llg} by~$\bphi\in
C^\infty(D_T;\R^3)$, using~$|\bm|=1$ and integration by parts, we
deduce
\[
\alpha\dual{\bm_t}{\bm\times\bphi}_{D_T} 
+
\dual{\bm\times\bm_t}{\bm\times\bphi}_{D_T}+C_e 
\dual{\nabla\bm}{\nabla(\bm\times\bphi)}_{D_T}
=
\dual{\bH}{\bm\times\bphi}_{D_T}.
\]

To tackle the eddy current equations on $\R^3$, we aim to employ
FE/BE coupling methods.
To that end, we employ the \emph{magnetic} approach
from~\cite{dual}, which eventually results in a variant of the
\emph{Trifou}-discretisation of the eddy-current Maxwell equations.
The magnetic approach is more or less mandatory in our case, since
the coupling with the LLG equation requires the magnetic field
rather than the electric field. 

Multiplying~\eqref{eq:MLLG2} by~$\bxi\in C^{\infty}(D,\R^3)$
satisfying~$\nabla\times\bxi=0$ in~$D^\ast$, integrating over~$\R^3$,
and using integration by parts, we obtain for almost
all~$t\in[0,T]$
\[
\mu_0
\dual{\bH_t(t)}{\bxi}_{\R^3}
+
\dual{\bE(t)}{\nabla\times\bxi}_{\R^3}
=
-\mu_0
\dual{\bm_t(t)}{\bxi}_{D}.
\]
Using~$\nabla\times\bxi=0$ in~$D^\ast$ and~\eqref{eq:MLLG1} we deduce
\[
\mu_0
\dual{\bH_t(t)}{\bxi}_{\R^3}
+
\sigma^{-1}\dual{\nabla\times\bH(t)}{\nabla\times\bxi}_{D}
=
-\mu_0
\dual{\bm_t(t)}{\bxi}_{D}.
\]
Since~$\nabla\times\bH=\nabla\times\bxi=0$ in~$D^\ast$, 
there exists~$\varphi$ and~$\zeta$ such
that~$\bH=\nabla\varphi$ and~$\bxi=\nabla\zeta$ in~$D^\ast$. 
Therefore, the above equation can be rewritten as
\[
\mu_0
\dual{\bH_t(t)}{\bxi}_{D}
+
\mu_0
\dual{\nabla\varphi_t(t)}{\nabla\zeta}_{D^\ast}
+
\sigma^{-1}\dual{\nabla\times\bH(t)}{\nabla\times\bxi}_{D}
=
-\mu_0
\dual{\bm_t(t)}{\bxi}_{D}.
\]
Since~\eqref{eq:MLLG3} implies~$\dive(\bH)=0$ in~$D^\ast$,
we have~$\Delta\varphi=0$ in~$D^\ast$, so that
(formally)~$\Delta\varphi_t=0$ in~$D^\ast$. Hence
integration by parts yields
\begin{equation}\label{eq:Ht}
\mu_0
\dual{\bH_t(t)}{\bxi}_{D}
-
\mu_0
\dual{\partial_n^+\varphi_t(t)}{\zeta}_{\Gamma}
+
\sigma^{-1}\dual{\nabla\times\bH(t)}{\nabla\times\bxi}_{D}
=
-\mu_0
\dual{\bm_t(t)}{\bxi}_{D},
\end{equation}
where~$\partial_n^+$ is the exterior Neumann trace operator with
the limit taken from~$D^\ast$. The advantage of the above
formulation is that no integration over the unbounded domain~$D^\ast$ is
required.
The exterior Neumann trace~$\partial_n^+\varphi_t$ can be computed
from the exterior Dirichlet trace~$\lambda$ of~$\varphi$ by using the
Dirichlet-to-Neumann operator~$\dtn$, which is defined as follows.

Let $\gamma^-$ be the interior Dirichlet trace
operator and $\partial_n^-$ be the interior normal derivative
or Neumann trace operator.
(The $-$ sign indicates the trace is taken from $D$.)
%Analogously, we define $\gamma^+$ and $\partial_n^+$ as the
%exterior counterparts (with the same outward pointing normal
%vector).
Recalling the fundamental solution of the Laplacian
$G(x,y):=1/(4\pi|x-y|)$, we introduce the following integral operators
defined formally on $\Gamma$ as 
\begin{align*}
\slp(\lambda):=\gamma^-\overline\slp(\lambda),
\quad
\dlp(\lambda):=\gamma^-\overline\dlp(\lambda)+\mfrac12,
\quad\text{and}\quad
\hyp(\lambda):=-\partial_n^-\overline\dlp(\lambda),
\end{align*}
where, for $x\notin\Gamma$,
\begin{align*}
\overline\slp(\lambda)(x)
:=
\int_{\Gamma} G(x,y) \lambda(y)\,ds_y
\quad\text{and}\quad 
\overline\dlp(\lambda)(x)
:=\int_{\Gamma} \partial_{n(y)}G(x,y)\lambda(y)\,ds_y.
\end{align*}
Moreover, let $\dlp^\prime$ denote the adjoint operator of $\dlp$
with respect to the extended $L^2$-inner product.
Then the exterior Dirichlet-to-Neumann map $\dtn\colon
H^{1/2}(\Gamma)\to H^{-1/2}(\Gamma)$ can be represented as
\begin{equation}\label{eq:dtn}
\dtn 
= 
- \slp^{-1}(1/2-\dlp).
\end{equation}
Another more symmetric representation is
\begin{equation}\label{eq:dtn2}
\dtn 
=
-(1/2-\dlp^\prime) \slp^{-1}(1/2-\dlp)-\hyp.
\end{equation}

Recall that~$\varphi$ satisfies~$\bH=\nabla\varphi$ in~$D^\ast$.
We can choose~$\varphi$ satisfying~$\varphi(x)=O(|x|^{-1})$ as
$|x|\to\infty$.
Now if~$\lambda=\gamma^+\varphi$ 
then~$\lambda_t=\gamma^+\varphi_t$.
Since~$\Delta\varphi=\Delta\varphi_t=0$ in~$D^\ast$, and since the
exterior Laplace problem has a unique solution 
we
have~$\dtn\lambda=\partial_n^+\varphi$
and~$\dtn\lambda_t=\partial_n^+\varphi_t$.
Hence~\eqref{eq:Ht} can be rewritten as
\begin{equation}\label{eq:H lam t}
\dual{\bH_t(t)}{\bxi}_{D}
-
\dual{\dtn\lambda_t(t)}{\zeta}_{\Gamma}
+
\mu_0^{-1}
\sigma^{-1}\dual{\nabla\times\bH(t)}{\nabla\times\bxi}_{D}
=
-
\dual{\bm_t(t)}{\bxi}_{D}.
\end{equation}
We remark that if~$\nabla_\Gamma$ denotes the surface gradient
operator on~$\Gamma$ then it is well-known that
$
\nabla_\Gamma\lambda
=
(\nabla\varphi)|_{\Gamma}
-
(\partial_n^+\varphi)\bn
=
\bH|_{\Gamma}
-
(\partial_n^+\varphi)\bn;
$
see e.g.~\cite[Section~3.4]{Monk03}. Hence~$\bn\times\nabla_\Gamma\lambda =
\bn\times\bH|_{\Gamma}$.

The above analysis prompts us to define the following
weak formulation.

\begin{definition}\label{def:fembemllg}
A triple $(\bm,\bH,\lambda)$ satisfying
\begin{align*}
\bm &\in \Hone{D_T}
\quad\text{and}\quad
\bm_t|_{\Gamma_T} \in L^2(0,T;H^{-1/2}(\Gamma)),
\\
\bH &\in L^2(0,T;\Hcurl{D})\cap H^1(0,T;\Ltwo{D}), 
\\
\lambda &\in H^1(0,T;H^{1/2}(\Gamma))
%\\
%\phi &\in L^2(0,T;H^{1/2}(\Gamma))
\end{align*}
is called a weak solution to~\eqref{eq:strong}--\eqref{eq:con}
if the following statements hold 
 \begin{enumerate}
  \item $|\bm|=1$ almost everywhere in $D_T$; \label{ite:1}

  \item $\bm(0,\cdot)=\bm^0$, $\bH(0,\cdot)=\bH^0$, and 
$\lambda(0,\cdot)=\gamma^+ \varphi^0$ where~$\varphi^0$ is a scalar
function satisfies $\bH^0=\nabla\varphi^0$ in~$D^\ast$ (the
assumption~\eqref{eq:ini} ensures the existence
of~$\varphi^0$); \label{ite:2}

  \item For all $\bphi\in C^\infty(D_T;\R^3)$
\label{ite:3}
  \begin{subequations}\label{eq:wssymm}
  \begin{align}
\alpha\dual{\bm_t}{\bm\times\bphi}_{D_T} 
&+
\dual{\bm\times\bm_t}{\bm\times\bphi}_{D_T}+C_e 
\dual{\nabla\bm}{\nabla(\bm\times\bphi)}_{D_T}
\nonumber
\\
&=
\dual{\bH}{\bm\times\bphi}_{D_T};
  \label{eq:wssymm1}
\end{align}

\item There holds $\bn\times \nabla_\Gamma\lambda = \bn\times
\bH|_{\Gamma}$ in the sense of traces; \label{ite:4}

\item For $\bxi\in C^\infty(D;\R^3)$ and $\zeta\in
C^\infty(\Gamma)$ satisfying
$\bn\times\bxi|_{\Gamma}=\bn\times\nabla_\Gamma\zeta$ in the sense of
traces \label{ite:5}
\begin{align}\label{eq:wssymm2}
\dual{\bH_t}{\bxi}_{D_T}
-\dual{\dtn\lambda_t}{\zeta}_{\Gamma_T}
+
\sigma^{-1}\mu_0^{-1}\dual{\nabla\times\bH}{\nabla\times\bxi}_{D_T}
&=-\dual{\bm_t}{\bxi}_{D_T};
\end{align}
\end{subequations}
%where $\SS$ denotes the exterior Dirichlet-to-Neumann operator defined by $\dtn v=\partial_n u$ for
%\begin{align*}
%-\Delta u =0\text{ in }D^\ast,\quad |u(x)|=\mathcal{O}(|x|^{-1})\;\text{ as }x\to\infty,\quad\text{and}\quad u|_{\Gamma}=v
%\end{align*}
\item For almost all $t\in[0,T]$
\label{ite:6}
\begin{gather}
 \norm{\nabla \bm(t)}{\Ltwo{D}}^2 
+
 \norm{\bH(t)}{\Hcurl{D}}^2
+
\norm{\lambda(t)}{H^{1/2}(\Gamma)}^2
\nonumber
\\
+ 
 \norm{\bm_t}{\Ltwo{D_t}}^2
+
 \norm{\bH_t}{\Ltwo{D_t}}^2
+
 \norm{\lambda_t}{H^{1/2}(\Gamma_t)}^2 \leq C,
\label{eq:energybound2}
 \end{gather}
where the constant $C>0$ is independent of $t$.
\end{enumerate}
\end{definition}

%We will derive the weak form of Definition~\ref{def:fembemllg} from
%the strong form~\eqref{eq:strong} in Subsection~\ref{subsec:equ
%def} below. 
The reason we integrate over~$[0,T]$ in~\eqref{eq:H lam t} to
have~\eqref{eq:wssymm2} is to facilitate the passing to the
limit in the proof of the main theorem.
The following lemma justifies the above definition.
\begin{lemma}\label{lem:equidef}
Let $(\bm,\bH,\bE)$ be a strong solution 
of~\eqref{eq:strong}--\eqref{eq:con}. If $\varphi\in
H(0,T;H^1(D^\ast))$
satisfies $\nabla\varphi=\bH|_{D^\ast_T}$, and if
$\lambda:=\gamma^+\varphi$, then the triple
$(\bm,\bH|_{D_T},\lambda)$
is a weak solution in the sense of Definition~\ref{def:fembemllg}. 

Conversely, let $(\bm,\bH,\lambda)$ be a sufficiently smooth
solution in the sense of Definition~\ref{def:fembemllg}, 
and let~$\varphi$ be the solution of
\begin{equation}\label{eq:var phi}
\Delta\varphi = 0
\text{ in } D^\ast,
\quad
\varphi = \lambda
\text{ on } \Gamma,
\quad
\varphi(x) = O(|x|^{-1})
\text{ as } |x|\to\infty.
\end{equation}
%\begin{alignat}{2}
%\Delta\varphi &= 0
%\quad && \text{in } D^\ast,
%\notag
%\\
%\varphi &= \lambda
%\quad && \text{on } \Gamma,
%\label{eq:var phi}
%\\
%\varphi(x) &= O(|x|^{-1})
%\quad && \text{as } |x|\to\infty.
%\notag
%\end{alignat}
Then $(\bm,\overline\bH,\bE)$ is a strong solution
to~\eqref{eq:strong}--\eqref{eq:con}, where~$\overline\bH$ is 
defined by
\begin{align}\label{eq:repform}
\overline\bH:=\begin{cases}
\bH &\quad\text{in }D_T,\\
%\overline{\dlp}(\nabla_\Gamma\lambda)-\overline{\slp}\dtn
%(\nabla_\Gamma\lambda) 
\nabla\varphi
&\quad\text{in } D_T^\ast,
\end{cases}
\end{align}
and~$\bE$ is reconstructed by letting
$\bE=\sigma^{-1}(\nabla\times\bH)$ in $D_T$ and by solving
\begin{subequations}\label{eq:el}
\begin{alignat}{2}
%\label{eq:el}
%\begin{split}
\nabla\times\bE &= -\mu_0\overline\bH_t
&&\quad\text{in } D_T^\ast,
\\
{\rm div}(\bE) &=0
&&\quad\text{in }D_T^\ast,
\\
\bn\times \bE|_{D_T^\ast} &= \bn\times \bE|_{D_T}
&&\quad\text{on }\Gamma_T.
%\end{split}
\end{alignat}
\end{subequations}
 \end{lemma}
\begin{proof}%[Proof of Lemma~\ref{lem:equidef}]
We follow~\cite{dual}.
Assume that $(\bm,\bH,\bE)$
satisfies~\eqref{eq:strong}--\eqref{eq:con}.
Then clearly Statements~\eqref{ite:1}, \eqref{ite:2} and~\eqref{ite:6} in
Definition~\ref{def:fembemllg} hold, noting~\eqref{eq:ini}. 
Statements~\eqref{ite:3}, \eqref{ite:4}
and~\eqref{ite:5} also hold due to the analysis above
Definition~\ref{def:fembemllg}.
The converse is also true due to the
well-posedness of~\eqref{eq:el} as stated
in~\cite[Equation~(15)]{dual}. 
\end{proof}
\begin{remark}
The solution~$\varphi$ to~\eqref{eq:var phi} can be represented
as
$
\varphi
=
(1/2+\dlp)\lambda - \slp\dtn\lambda.
$
\end{remark}

The next subsection defines the spaces and functions to be used
in the approximation of the weak solution
the sense of Definition~\ref{def:fembemllg}.
\subsection{Discrete spaces and functions}\label{subsec:dis spa}
For time discretisation, we use a uniform partition $0\leq t_i\leq
T$, $i=0,\ldots,N$ with $t_i:=ik$ and $k:=T/N$. The spatial
discretisation is determined by a (shape) regular triangulation
$\TT_h$ of $D$ into compact tetrahedra $T\in\TT_h$ with diameter
$h_T/C\leq h\leq Ch_T$ for some uniform constant $C>0$.
Denoting by~$\NN_h$ the set of nodes of~$\TT_h$, 
we define the following spaces
\begin{align*}
 \SS^1(\TT_h)&:=\set{\phi_h\in C(D)}{\phi_h|T \in \PP^1(T)\text{ for all } T\in\TT_h},\\
 \KK_{\bphi_h}&:=\set{\bpsi_h\in\SS^1(\TT_h)^3}{\bpsi_h(z)\cdot\bphi_h(z)=0\text{
for all }z\in\NN_h},
\quad\bphi_h\in\SS^1(\TT_h)^3,
\end{align*}
where $\PP^1(T)$ is the space of polynomials of degree at most~1
on~$T$.

For the discretisation of~\eqref{eq:wssymm2}, we employ the
space $\NN\DD^1(\TT_h)$ of
first order N\'ed\'elec (edge) elements for $\bH$ and 
and the space $\SS^1(\TT_h|_\Gamma)$ for $\lambda$.
Here $\TT_h|_{\Gamma}$ denotes the restriction of the
triangulation to the boundary~$\Gamma$.
It follows from Statement~\ref{ite:4} in
Definition~\ref{def:fembemllg} that for each~$t\in[0,T]$, the
pair~$(\bH(t),\lambda(t))\in\XX$. We approximate the space~$\XX$ by
\begin{align*}
\XX_h:=\set{(\bxi,\zeta)\in \NN\DD^1(\TT_h)\times
\SS^1(\TT_h|_\Gamma)}{\bn\times \nabla_\Gamma\zeta =
\bn\times\bxi|_\Gamma}.
\end{align*}
To ensure the condition $\bn\times \nabla_\Gamma\zeta = \bn\times
\bxi|_{\Gamma}$, we observe the following.
For any~$\zeta\in \SS^1(\TT_h|_\Gamma)$, if $e$ denotes an edge of
$\TT_h$
on $\Gamma$, then $\int_e \bxi\cdot \btau\,ds = 
\int_e \nabla \zeta \cdot \btau\,ds = \zeta(z_0)-\zeta(z_1)$, where
$\btau$ is the unit direction vector on~$e$, and
$z_0,z_1$ are the endpoints of $e$.
Thus, taking as degrees of freedom all interior edges of $\TT_h$
(i.e. $\int_{e_i}\bxi\cdot\btau \,ds$) as well as all nodes of
$\TT_h|_\Gamma$ (i.e.~$\zeta(z_i)$), we fully determine
a function pair $(\bxi,\zeta)\in\XX_h $.
Due to the considerations above, it is clear that the above
space can be implemented directly without use of Lagrange
multipliers or other extra equations. 

The density properties of the finite element
spaces~$\{\XX_h\}_{h>0}$
are shown in Subsection~\ref{subsec:som lem}; see 
Lemma~\ref{lem:den pro}.

Given functions 
$\bw_h^i\colon D\to \R^d$, $d\in\N$,
for all $i=0,\ldots,N$ we define for all $t\in [t_i,t_{i+1}]$
\begin{align*}
\bw_{hk}(t):=\frac{t_{i+1}-t}{k}\bw_h^i + \frac{t-t_i}{k}\bw_h^{i+1},
\quad
\bw_{hk}^-(t):=\bw_h^i,
\quad
\bw_{hk}^+(t):=\bw_h^{i+1} .
\end{align*}
Moreover, we define
\begin{equation}\label{eq:dt}
d_t\bw_h^{i+1}:=\frac{\bw_h^{i+1}-\bw_h^i}{k}\quad\text{for all }
i=0, \ldots, N-1.
\end{equation}

Finally, we denote by $\Pi_{\SS}$ the usual interpolation operator on
$\SS^1(\TT_h)$.%, and $\Pi_{\NN\DD}$ and $\Pi_{\RR\TT}$ the canonical interpolation operators onto $\NN\DD^1(\TT_h)$ (cf.~\cite{ciarlet}) and onto $\RR\TT_0^0(\TT_h|_\Gamma)$ (cf.~\cite[Proposition~3.6]{BrezziFortin}), respectively.
We are now ready to present the algorithm to compute approximate
solutions to problem~\eqref{eq:strong}--\eqref{eq:con}.

\subsection{Numerical algorithm}\label{section:alg}
In the sequel, when there is no confusion we use the same
notation $\bH$ for 
%we denote by $\bH=\bH|_{D_T}\colon D_T\to \R^3$ 
the restriction of $\bH\colon \R^3_T\to \R^3$ to the domain $D_T$.
\begin{algorithm}\label{algorithm}
\mbox{}

 \textbf{Input:} Initial data $\bm^0_h\in\SS^1(\TT_h)^3$, 
$(\bH^0_h,\lambda_h^0)\in\XX_h$, 
and parameter $\theta\in[0,1]$.

 \textbf{For} $i=0,\ldots,N-1$ \textbf{do:}
 \begin{enumerate}
  \item Compute the unique function $\bv^i_h\in \KK_{\bm_h^i}$
satisfying for all $\bphi_h\in \KK_{\bm_h^i}$
  \begin{align}\label{eq:dllg}
  \begin{split}
   \alpha\dual{\bv_h^i}{\bphi_h}_D 
&+ 
\dual{\bm_h^i\times \bv_h^i}{\bphi_h}_D
+ 
C_e\theta k \dual{\nabla \bv_h^i}{\nabla \bphi_h}_D
\\
&=
-C_e \dual{\nabla \bm_h^i}{\nabla \bphi_h}_D 
+
\dual{\bH_h^i}{\bphi_h}_D.
  \end{split}
  \end{align}
  \item Define $\bm_h^{i+1}\in \SS^1(\TT_h)^3$ nodewise by 
\begin{equation}\label{eq:mhip1}
\bm_h^{i+1}(z) =\bm_h^i(z) + k\bv_h^i(z) \quad\text{for all } z\in\NN_h.
\end{equation}
\item Compute the unique functions
$(\bH_h^{i+1},\lambda_h^{i+1})\in\XX_h$
 satisfying for all
$(\bxi_h,\zeta_h)\in \XX_h$
\begin{align}\label{eq:dsymm}
\dual{d_t\bH_h^{i+1}}{\bxi_h}_{D}
&
-\dual{d_t\dtn_h\lambda_h^{i+1}}{\zeta_h}_\Gamma
+
\sigma^{-1}\mu_0^{-1}\dual{\nabla\times\bH_h^{i+1}}{\nabla\times\bxi_h}_{D}
\nonumber
\\
&=
-\dual{\bv_h^i}{\bxi_h}_{D},
\end{align}
where $\dtn_h\colon H^{1/2}(\Gamma)\to \SS^1(\TT_h|_\Gamma)$ 
is the discrete Dirichlet-to-Neumann operator to be defined
later.
\end{enumerate}

\textbf{Output:} Approximations $(\bm_h^i,\bH_h^i,\lambda_h^i)$
for all $i=0,\ldots,N$.
\end{algorithm}

%The idea to use the linear formula~\eqref{eq:mhip1} was introduced
The linear formula~\eqref{eq:mhip1} was introduced
in~\cite{bartels} and used in~\cite{Abert_etal}.
Equation~\eqref{eq:dsymm} requires the computation
of~$\dtn_h\lambda$ for any~$\lambda\in H^{1/2}(\Gamma)$.
%of~$\dual{\dtn_h\lambda}{\zeta_h}_{\Gamma}$ for~$\lambda\in
%H^{1/2}(\Gamma)$ and~$\zeta_h\in\SS^1(\TT_h|_\Gamma)$.
This is done by use of
the boundary element method. 
%in two different ways.
Let~$\mu\in H^{-1/2}(\Gamma)$ 
and~$\mu_h\in\PP^0(\TT_h|_\Gamma)$
be, respectively, the solution of
%\begin{equation}\label{eq:VK}
%\slp\mu= (\dlp-1/2)\lambda
%\end{equation}
%and
\begin{align}\label{eq:bem}
\slp\mu= (\dlp-1/2)\lambda
\quad\text{and}\quad
\dual{\slp \mu_h}{\nu_h}_\Gamma =
\dual{(\dlp-1/2)\lambda}{\nu_h}_\Gamma\quad\forall\nu_h\in
 \PP^0(\TT_h|_\Gamma),
\end{align}
where $\PP^0(\TT_h|_\Gamma)$ is the space of 
piecewise-constant functions on~$\TT_h|_\Gamma$.

If the representation~\eqref{eq:dtn} of~$\dtn$ is used, 
then~$\dtn\lambda=\mu$, and
we can uniquely define~$\dtn_h\lambda$
by solving
\begin{equation}\label{eq:bem1}
\dual{\dtn_h\lambda}{\zeta_h}_\Gamma
=
\dual{\mu_h}{\zeta_h}_\Gamma
\quad\forall\zeta_h\in\SS^1(\TT_h|_\Gamma).
\end{equation}
This is known as the Johnson-N\'ed\'elec coupling. 

If we use the representation~\eqref{eq:dtn2} for~$\dtn\lambda$
then $\dtn\lambda = (1/2-\dlp^\prime)\mu-\hyp\lambda$.
In this case we can uniquely define~$\dtn_h\lambda$ by solving
\begin{align}\label{eq:bem2}
\dual{\dtn_h\lambda}{\zeta_h}_\Gamma 
= 
\dual{(1/2-\dlp^\prime)\mu_h}{\zeta_h}_\Gamma 
-
\dual{\hyp \lambda}{\zeta_h}_\Gamma
\quad\forall\zeta_h\in \SS^1(\TT_h|_\Gamma).
\end{align}
This approach yields an (almost) symmetric system and is called
Costabel's coupling.

In practice,~\eqref{eq:dsymm} only requires the computation
of~$\dual{\dtn_h\lambda_h}{\zeta_h}_\Gamma$ for
any~$\lambda_h, \zeta_h\in\SS^1(\TT_h|_\Gamma)$. 
So in the implementation,
neither~\eqref{eq:bem1} nor~\eqref{eq:bem2} has to be solved. 
It suffices to solve the second equation in~\eqref{eq:bem}
and compute the right-hand side of either~\eqref{eq:bem1}
or~\eqref{eq:bem2}.

It is proved in~\cite[Appendix~A]{afembem} that Costabel's coupling
results in a discrete operator which is uniformly elliptic and
continuous:
\begin{align}\label{eq:dtnelliptic}
\begin{split}
-\dual{\dtn_h \zeta_h}{\zeta_h}_\Gamma&\geq
C_\dtn^{-1}\norm{\zeta_h}{H^{1/2}(\Gamma)}^2\quad\text{for all
}\zeta_h\in \SS^1(\TT_h|_\Gamma),\\
\norm{\dtn_h\zeta}{H^{-1/2}(\Gamma)}^2&\leq C_\dtn
\norm{\zeta}{H^{1/2}(\Gamma)}^2\quad\text{for all }\zeta\in H^{1/2}(\Gamma),
\end{split}
\end{align}
for some constant $C_\dtn>0$ which depends only on $\Gamma$. 
Even though the remainder of the analysis works analogously for both
approaches,
we are not aware of an ellipticity result of the 
form~\eqref{eq:dtnelliptic} for the Johnson-N\'ed\'elec
approach. Thus, from now on $\dtn_h$ is understood to be defined
by~\eqref{eq:bem2}.

%\begin{remark} %\mbox{}
%\begin{enumerate}
%\item
%In Section~\ref{section:weak} (Lemma~\ref{lem:bil for})
%we will show that Algorithm~\ref{algorithm} is well-defined.
%\item
%The idea to use the linear formula~\eqref{eq:mhip1} was introduced
%in~\cite{bartels} and used in~\cite{Abert_etal}.
%\end{enumerate}
%\end{remark}
%%%%%%%%%%%%%%%%%%%%%%%%%%%
\subsection{Main result}
Before stating the main result of this part of the paper, 
we first state some general
assumptions. Firstly, the weak convergence of approximate solutions
requires the following
conditions on $h$ and $k$, depending on the value of the parameter~$\theta$
in~\eqref{eq:dllg}:
\begin{equation}\label{eq:hk 12}
\begin{cases}
k = o(h^2) \quad & \text{when } 0 \le \theta < 1/2, \\
k = o(h)   \quad & \text{when } \theta = 1/2, \\
\text{no condition} & \text{when } 1/2 < \theta \le 1.
\end{cases}
\end{equation}
Some supporting lemmas which have their own interests do not require any
condition when $\theta=1/2$. For those results, a slightly different
condition is required, namely
 \begin{equation}\label{eq:hk con}
\begin{cases}
k = o(h^2) \quad & \text{when } 0 \le \theta < 1/2, \\
\text{no condition} & \text{when } 1/2 \le \theta \le 1.
\end{cases}
\end{equation}
The initial data are assumed to satisfy
\begin{equation}\label{eq:mh0 Hh0}
\sup_{h>0}
\left(
\norm{\bm_h^0}{H^1(D)}
+
\norm{\bH_h^0}{\Hcurl{D}}
+
\norm{\lambda_h^0}{H^{1/2}(\Gamma)}
\right)
<\infty
\quad\text{and}\quad
\lim_{h\to0}
\norm{\bm_h^0-\bm^0}{\Ltwo{D}}
= 0.
\end{equation}
% Finally we assume that the meshes~\{$\TT_h\}_{h>0}$ are defined
% such that
% \begin{equation}\label{eq:Th}
% \PP^0(\TT_{h'}|_\Gamma)
% \subset
% \PP^0(\TT_{h}|_\Gamma)
% \quad\text{and}\quad
% \SS^1(\TT_{h'}|_\Gamma)
% \subset
% \SS^1(\TT_{h}|_\Gamma)
% \quad\text{for } h<h'.
% \end{equation}
% This assumption is only needed in the proof of
% Lemma~\ref{lem:wea con} below.

We are now ready to state the main result of this part of the paper.
\begin{theorem}[Existence of solutions]\label{thm:weakconv}
Under the assumptions~\eqref{eq:hk 12} and~\eqref{eq:mh0
Hh0}, the problem~\eqref{eq:strong}--\eqref{eq:con} has a 
solution~$(\bm,\bH,\lambda)$ in the sense of
Definition~\ref{def:fembemllg}.
\end{theorem}
\section{Proofs of the main result}\label{sec:pro}
%\subsection{Derivation of the weak from in Definition~\ref{def:fembemllg}}\label{subsec:equ def}
\subsection{Some lemmas}\label{subsec:som lem}
In this subsection we prove all important lemmas which are directly related
to the proofs of the theorem.
The first lemma proves density properties of the discrete spaces.
\begin{lemma}\label{lem:den pro}
Provided that the meshes $\{\TT_h\}_{h>0}$ are regular, 
the union~$\bigcup_{h>0}\XX_h$ is dense in~$\XX$.
%there holds
%\[
%\overline{\bigcup_{h>0}\XX_h}
%=
%\XX
%\]
%where the closure is taken in the $\XX$-topology. 
Moreover, there exists an interpolation operator
$\Pi_\XX:=(\Pi_{\XX,D},\Pi_{\XX,\Gamma})\colon 
\big(\H^2(D)\times H^2(\Gamma)\big)
\cap \XX\to \XX_h$ which satisfies
\begin{align}\label{eq:int}
\norm{(1-\Pi_\XX)(\bxi,\zeta)}{\Hcurl{D}\times H^{1/2}(\Gamma) }
&\leq C_{\XX} h( \norm{\bxi}{\H^2(D)}+ h^{1/2}\norm{\zeta}{H^2(\Gamma)}),
\end{align}
where $C_\XX>0$ depends only on $D$, $\Gamma$, and the shape
regularity of $\TT_h$.
\end{lemma}
\begin{proof}
%The density of ${\bigcup_{h>0}\SS^1(\TT_h|_\Gamma)}$ in
%$H^{1/2}(\Gamma)$ is well-known, and the density of the set
%${\bigcup_{h>0}\NN\DD^1(\TT_H)}$ in
%$\Hcurl{D}$  is a result of~\cite[Theorem~8.1]{hipt}
%and the density of~$\Hone{D}$ in~$\Ltwo{D}$. 
The interpolation operator
$\Pi_\XX:=(\Pi_{\XX,D},\Pi_{\XX,\Gamma})\colon 
\big(\H^2(D)\times H^2(\Gamma)\big)
\cap \XX\to \XX_h$ is constructed as follows.
The interior degrees of freedom (edges) of $\Pi_\XX(\bxi,\zeta)$
are equal to the interior degrees of freedom of
$\Pi_{\NN\DD}\bxi\in\NN\DD^1(\TT_h)$, where $\Pi_{\NN\DD}$ is
the usual interpolation operator onto $\NN\DD^1(\TT_h)$.
The degrees of freedom of $\Pi_\XX(\bxi,\zeta)$ which lie on
$\Gamma$ (nodes) are equal to $\Pi_\SS\zeta$.
By the definition of $\XX_h$, this fully determines $\Pi_\XX$.
Particularly, since $\bn\times
\bxi|_\Gamma=\bn\times\nabla_\Gamma\zeta$, there holds
$\Pi_{\NN\DD}\bxi|_\Gamma = \Pi_{\XX,\Gamma}(\bxi,\zeta)$.
Hence, the interpolation error can be bounded by
\begin{align*}
\norm{(1-\Pi_\XX)(\bxi,\zeta)}{\Hcurl{D}\times H^{1/2}(\Gamma) }&\leq \norm{(1-\Pi_{\NN\DD})\bxi}{\Hcurl{D}}+\norm{(1-\Pi_\SS)\zeta}{H^{1/2}(\Gamma)}\\
&\lesssim h( \norm{\bxi}{\H^2(D)}+ h^{1/2}\norm{\zeta}{H^2(\Gamma)}).
\end{align*}
Since $\big(\H^2(D)\times H^2(\Gamma)\big)\cap \XX$ is dense in
$\XX$, this concludes the proof.
\end{proof}

The following lemma gives an equivalent form
to~\eqref{eq:wssymm2} and shows that 
Algorithm~\ref{algorithm} is well-defined.
\begin{lemma}\label{lem:bil for}
Let $a(\cdot,\cdot)\colon \XX\times\XX\to \R$, $a_h(\cdot,\cdot)\colon \XX_h\times\XX_h\to \R$, and
$b(\cdot,\cdot)\colon \Hcurl{D}\times\Hcurl{D}\to \R$ be
bilinear forms defined by
\begin{align*}
a(A,B)&:=\dual{\bpsi}{\bxi}_D -\dual{\dtn \eta}{\zeta}_\Gamma,
\\
a_h(A_h,B_h)
&:=
\dual{\bpsi_h}{\bxi_h}_D
-
\dual{\dtn_h\eta_h}{\zeta_h}_{\Gamma},
\\
b(\bpsi,\bxi)&:=
\sigma^{-1}\mu_0^{-1}
\dual{\nabla\times \bpsi}{\nabla\times \bxi}_\Gamma,
\end{align*}
%\begin{align*}
%\begin{split}
%a(A,B)&:=\dual{\bH}{\bxi}_D -\dual{\dtn \eta}{\zeta}_\Gamma,\\
%b(\bH,\bxi)&:=\dual{\nabla\times \bH}{\nabla\times \bxi}_\Gamma,
%\end{split}
%\end{align*}
for all $\bpsi, \bxi\in\Hcurl{D}$, $A:=(\bpsi,\eta)$, $B:=(\bxi,\zeta)\in\XX$, $A_h = (\bpsi_h,\eta_h), B_h = (\bxi_h,\zeta_h) \in \XX_h$.
Then 
\begin{enumerate}
\item
The bilinear forms satisfy,
for all 
$A=(\bpsi,\eta)\in\XX$ and
$A_h=(\bpsi_h,\eta_h)\in\XX_h$,
\begin{align}\label{eq:elliptic}
\begin{split}
a(A,A)&\geq C_{\rm ell} 
\big(
\norm{\bpsi}{\Ltwo{D}}^2+\norm{\eta}{H^{1/2}(\Gamma)}^2
\big),
\\
a_h(A_h,A_h)
&\geq 
C_{\rm ell}
\big(
\norm{\bpsi_h}{\Ltwo{D}}^2+\norm{\eta_h}{H^{1/2}(\Gamma)}^2
\big),
\\
b(\bpsi,\bpsi)
&\geq C_{\rm ell}\norm{\nabla\times
\bpsi}{\Ltwo{D}}^2.
\end{split}
\end{align}
\item
Equation~\eqref{eq:wssymm2} is equivalent to
\begin{equation}\label{eq:bil for}
\int_0^T
a(A_t(t),B) \, dt
%a(\partial_t A(t),B) \, dt
+
\int_0^T
b(\bH(t),\bxi) \, dt
=
-
\dual{\bm_t}{\bxi}_{D_T}
%\int_0^T
%a(\partial_t A,B)+
%b(A,B) \,dt
%=
%-
%\dual{\bm_t}{\bxi}_{D_T}
\end{equation}
for all~$B=(\bxi,\zeta)\in \XX$, where~$A=(\bH,\lambda)$.
%for all~$B\in L^2(0,T;\XX)$.
\item
Equation~\eqref{eq:dsymm} is of the form
\begin{align}\label{eq:eddygeneral}
a_h(d_t A_h^{i+1},B_h) + b(\bH_h^{i+1},\bxi_h)
=
-\dual{\bv_h^i}{\bxi_h}_\Gamma
\end{align}
where $A_h^{i+1}:=(\bH_h^{i+1},\lambda_h^{i+1})$ and
$B_h:=(\bxi_h,\zeta_h)$.
\item
Algorithm~\ref{algorithm} is well-defined in the sense
that~\eqref{eq:dllg} and~\eqref{eq:dsymm} have unique solutions.
\end{enumerate}
\end{lemma}
\begin{proof}
The unique solvability of~\eqref{eq:dsymm} follows immediately from the continuity and ellipticity of the bilinear forms $a_h(\cdot,\cdot)$ 
and $b(\cdot,\cdot)$.

The unique solvability of~\eqref{eq:dllg} follows from the positive
definiteness of the left-hand side, the linearity of the right-hand side,
and the finite space dimension.
\end{proof}

The following lemma establishes an energy bound for the discrete
solutions.
\begin{lemma}\label{lem:denergygen}
Under the assumptions~\eqref{eq:hk con} and~\eqref{eq:mh0
Hh0}, there holds
for all $k<2\alpha$ and $j=1,\ldots,N$
\begin{align}\label{eq:denergy}
\sum_{i=0}^{j-1}&
\left(
\norm{\bH_h^{i+1}-\bH_h^i}{\Ltwo{D}}^2+\norm{\lambda_h^{i+1}-\lambda_h^i}{H^{1/2}(\Gamma)}^2
\right)\notag
\\
&+
k
\sum_{i=0}^{j-1}
\norm{\nabla\times \bH_h^{i+1}}{\Ltwo{D}}^2
+ \norm{\bH_h^j}{\Hcurl{D}}^2
+ \norm{\lambda_h^{j}}{H^{1/2}(\Gamma)}^2
+
\norm{\nabla \bm_h^{j}}{\Ltwo{D}}^2\notag
\\
&+
\max\{2\theta-1,0\}k^2
\sum_{i=0}^{j-1}
\norm{\nabla\bv_h^i}{\Ltwo{D}}^2
+
k
\sum_{i=0}^{j-1}
\norm{\bv_h^i}{\Ltwo{D}}^2
\\
&+k \sum_{i=0}^{j-1} (\norm{d_t\bH_h^{i+1}}{\Ltwo{D}}^2 +\norm{d_t\lambda_h^{i+1}}{H^{1/2}(\Gamma)}^2)+\sum_{i=0}^{j-1} 
\norm{\nabla\times (\bH^{i+1}_h-\bH^{i}_h)}{\Ltwo{D}}^2
\leq C_{\rm ener}.\notag
\end{align}
\end{lemma}
\begin{proof}
Choosing $B_h=A_h^{i+1}$ in~\eqref{eq:eddygeneral} and
multiplying the resulting equation by $k$ we obtain
\begin{equation}\label{eq:e1}
a_h(A_h^{i+1}-A_h^i,A_h^{i+1})
+
kb(\bH_h^{i+1},\bH_h^{i+1})
= -k\dual{\bv_h^i}{\bH_h^i}_D
 -k\dual{\bv_h^i}{\bH_h^{i+1}-\bH_h^i}_D.
 \end{equation}
On the other hand, it follows from~\eqref{eq:mhip1}
and~\eqref{eq:dllg} that
\begin{align*}%\label{eq:e2}
\begin{split}
\norm{\nabla \bm_h^{i+1}}{\Ltwo{D}}^2&=\norm{\nabla \bm_h^{i}}{\Ltwo{D}}^2+k^2\norm{\nabla \bv_h^{i}}{\Ltwo{D}}^2+2k\dual{\nabla\bm_h^i}{\nabla\bv_h^i}_D
\\
 & = \norm{\nabla\bm_h^i}{\Ltwo{D}}^2-2(\theta-\tfrac12)k^2\norm{\nabla\bv_h^i}{\Ltwo{D}}^2
%\\
% &\quad
- \frac{2\alpha k}{C_e}\norm{\bv_h^i}{\Ltwo{D}}^2 +
\frac{2k}{C_e}\dual{\bH_h^i}{\bv_h^i}_D,
\end{split}
 \end{align*}
which implies
\begin{align*}
k \dual{\bv_h^i}{\bH_h^i}_D
&=
\frac{C_e}{2}
\left(
\norm{\nabla\bm_h^{i+1}}{\Ltwo{D}}^2
-
\norm{\nabla\bm_h^{i}}{\Ltwo{D}}^2
\right)
%\\
%&\quad
+
(\theta-\tfrac12)k^2C_e \norm{\nabla\bv_h^i}{\Ltwo{D}}^2
+
\alpha k \norm{\bv_h^i}{\Ltwo{D}}^2.
\end{align*}
Inserting this into the first term on the right-hand side 
of~\eqref{eq:e1} and rearranging the resulting
equation yield, for any $\epsilon>0$,
%The combination of the last two estimates yields, after multiplication of~\eqref{eq:e1} by $2/(C_e r)$, the following
\begin{align*}
&a_h(A_h^{i+1}-A_h^i,A_h^{i+1})
+ 
%\frac{2k}{C_e r}
k
b(\bH_h^{i+1},\bH_h^{i+1})
\nonumber
\\
&+
\frac{C_e}{2}
\left(
\norm{\nabla \bm_h^{i+1}}{\Ltwo{D}}^2
-
\norm{\nabla\bm_h^i}{\Ltwo{D}}^2
\right)
+
(\theta-1/2)k^2 C_e
\norm{\nabla\bv_h^i}{\Ltwo{D}}^2
% +\frac{2\alpha k}{C_e}
+
\alpha k
\norm{\bv_h^i}{\Ltwo{D}}^2
\nonumber
\\
&=
%\frac{2k}{C_e}
-k
\dual{\bv_h^i}{\bH_h^{i+1}-\bH_h^{i}}_D
\\
&\leq
\frac{\epsilon k}{2}
\norm{\bv_h^i}{\Ltwo{D}}^2
+
\frac{k}{2\epsilon}
\norm{\bH_h^{i+1}-\bH_h^i}{\Ltwo{D}}^2
%\\
%&
\leq
\frac{\epsilon k}{2}
\norm{\bv_h^i}{\Ltwo{D}}^2
+
\frac{k}{2\epsilon}
a_h(A_h^{i+1}-A_h^i,A_h^{i+1}-A_h^i),
\end{align*}
where in the last step we used the definition of $a_h(\cdot,\cdot)$
and~\eqref{eq:dtnelliptic}.
Rearranging gives
\begin{align*}
a_h(A_h^{i+1}-A_h^i,A_h^{i+1})
&+ 
%\frac{2k}{C_e r}
k
b(\bH_h^{i+1},\bH_h^{i+1})
%\nonumber
%\\
%&+
+
\frac{C_e}{2}
\left(
\norm{\nabla \bm_h^{i+1}}{\Ltwo{D}}^2
-
\norm{\nabla\bm_h^i}{\Ltwo{D}}^2
\right)
\\
&+
(\theta-1/2)k^2 C_e
\norm{\nabla\bv_h^i}{\Ltwo{D}}^2
% +\frac{2\alpha k}{C_e}
+
(\alpha-\epsilon /2) k
\norm{\bv_h^i}{\Ltwo{D}}^2
\nonumber
\\
&\leq
\frac{k}{2\epsilon}
a_h(A_h^{i+1}-A_h^i,A_h^{i+1}-A_h^i).
%\norm{\bH_h^{i+1}-\bH_h^i}{\Ltwo{D}}^2.
\end{align*}
Summing over $i$ from $0$ to $j-1$ and 
(for the first term on the left-hand side)
applying Abel's summation by parts formula
\begin{equation}\label{equ:Abe}
\sum_{i=0}^{j-1}
(u_{i+1}-u_{i})u_{i+1}
=
\frac{1}{2} |u_j|^2
-
\frac{1}{2} |u_0|^2
+
\frac{1}{2}
\sum_{i=0}^{j-1}
|u_{i+1}-u_{i}|^2,
\end{equation}
we deduce, after multiplying the equation by two and rearranging,
\begin{align*}%\label{eq:e3}
%\frac{1}{C_e r}
&(1-k/\epsilon)
\sum_{i=0}^{j-1}
a_h(A_h^{i+1}-A_h^i,A_h^{i+1}-A_h^i)
+
2k
\sum_{i=0}^{j-1}
b(\bH_h^{i+1},\bH_h^{i+1})
%\frac{1}{C_e r}
+
a_h(A_h^j,A_h^j)
\\
&
+
 C_e
\norm{\nabla \bm_h^{j}}{\Ltwo{D}}^2
+
(2\theta-1)k^2  C_e
\sum_{i=0}^{j-1}
\norm{\nabla\bv_h^i}{\Ltwo{D}}^2
%\\
%&
+
%\frac{2\alpha k}{C_e}
(2\alpha-\epsilon ) k
\sum_{i=0}^{j-1}
\norm{\bv_h^i}{\Ltwo{D}}^2
\\
&
\leq 
%\frac{\eps k}{C_e}\sum_{i=0}^{j-1}\norm{\bv_h^i}{\Ltwo{D}}^2 
%\frac{kr}{\eps}
%\sum_{i=0}^{j-1}
%\norm{\bH_h^i-\bH_h^{i+1}}{\Ltwo{D}}^2
%+
%\frac{1}{C_e r}
 C_e
\norm{\nabla \bm_h^{0}}{\Ltwo{D}}^2
+
a_h(A_h^0,A_h^0).
\end{align*}
Since $k<2\alpha$ we can choose $\eps>0$ such
that $2\alpha-\epsilon >0$ and $1-k/\epsilon>0$. By noting the
ellipticity~\eqref{eq:dtnelliptic}, the bilinear forms
$a_h(\cdot,\cdot)$ and $b(\cdot,\cdot)$ are elliptic in their
respective (semi-)norms. We obtain
\begin{align}\label{eq:e4}
%\begin{split}
\sum_{i=0}^{j-1}&
\left(
\norm{\bH_h^{i+1}-\bH_h^i}{\Ltwo{D}}^2+\norm{\lambda_h^{i+1}-\lambda_h^i}{H^{1/2}(\Gamma)}^2
\right)
%\\
%&
+
k
\sum_{i=0}^{j-1}
\norm{\nabla\times \bH_h^{i+1}}{\Ltwo{D}}^2
+ \norm{\bH_h^j}{\Ltwo{D}}^2
\notag
\\
&
+ \norm{\lambda_h^{j}}{H^{1/2}(\Gamma)}^2
+
\norm{\nabla \bm_h^{j}}{\Ltwo{D}}^2
%\\
%&
+
(2\theta-1)k^2
\sum_{i=0}^{j-1}
\norm{\nabla\bv_h^i}{\Ltwo{D}}^2
+
k
\sum_{i=0}^{j-1}
\norm{\bv_h^i}{\Ltwo{D}}^2
\notag
\\
& %\qquad\qquad
\leq 
C
\left(
\norm{\nabla \bm_h^{0}}{\Ltwo{D}}^2
+
\norm{\bH_h^0}{\Ltwo{D}}^2
+ \norm{\lambda_h^{0}}{H^{1/2}(\Gamma)}^2
\right)
\leq
C,
%\end{split}
\end{align}
where in the last step we used~\eqref{eq:mh0 Hh0}.

It remains to consider the last three terms on the
left-hand side of~\eqref{eq:denergy}.
Again, we consider~\eqref{eq:eddygeneral} and select
$B_h=d_tA_h^{i+1}$ %as a test function 
to obtain after multiplication by~$2k$
\begin{align*}
 \begin{split}
 2k a_h(d_tA_h^{i+1},d_tA_h^{i+1})
&
+ 
2b(\bH_h^{i+1},\bH_h^{i+1}-\bH_h^i)
\\
&= -2k\dual{\bv_h^i}{d_t\bH_h^{i+1}}_D
\le
 k \norm{\bv_h^i}{\Ltwo{D}}^2
+
k \norm{d_t\bH_h^{i+1}}{\Ltwo{D}}^2,
\end{split}
\end{align*}
so that, noting~\eqref{eq:e4} and~\eqref{eq:elliptic},
\begin{align}\label{eq:e5}
\begin{split}
k \sum_{i=0}^{j-1} 
\left(
\norm{d_t\bH_h^{i+1}}{\Ltwo{D}}^2 
\right.
&+
\left.
\norm{d_t\lambda_h^{i+1}}{H^{1/2}(\Gamma)}^2
\right)
+  
2 \sum_{i=0}^{j-1} 
b(\bH_h^{i+1},\bH_h^{i+1}-\bH_h^i) 
\\
&
\lesssim
 k 
\sum_{i=0}^{j-1} 
\norm{\bv_h^i}{\Ltwo{D}}^2
\leq
C.
\end{split}
\end{align}
Using Abel's summation by parts formula~\eqref{equ:Abe} for the 
second sum on the left-hand side, and noting
%can be rewritten as
%\begin{align*}
%2 \sum_{i=0}^{j-1} b(\bH_h^{i+1},\bH_h^{i+1}-\bH_h^i)
%=
%b(\bH_h^{j},\bH_h^{j})
%-
%b(\bH_h^{0},\bH_h^0)
%+
%\sum_{i=0}^{j-1} 
%b(\bH_h^{i+1}-\bH_h^i,\bH_h^{i+1}-\bH_h^i).
%\end{align*}
%Inserting this into~\eqref{eq:e5}, using 
the ellipticity of 
the bilinear form $b(\cdot,\cdot)$ and~\eqref{eq:mh0 Hh0}, we obtain
together with~\eqref{eq:e4}
\begin{align}\label{eq:d4}
\sum_{i=0}^{j-1}&
(\norm{\bH_h^{i+1}-\bH_h^i}{\Ltwo{D}}^2+\norm{\lambda_h^{i+1}-\lambda_h^i}{H^{1/2}(\Gamma)}^2)\notag
\\
&+
k
\sum_{i=0}^{j-1}
\norm{\nabla\times \bH_h^{i+1}}{\Ltwo{D}}^2
+ \norm{\bH_h^j}{\Hcurl{D}}^2
+ \norm{\lambda_h^{j}}{H^{1/2}(\Gamma)}^2
+
\norm{\nabla \bm_h^{j}}{\Ltwo{D}}^2\notag
\\
&+
(2\theta-1)k^2
\sum_{i=0}^{j-1}
\norm{\nabla\bv_h^i}{\Ltwo{D}}^2
+
k
\sum_{i=0}^{j-1}
\norm{\bv_h^i}{\Ltwo{D}}^2
\\
&+k \sum_{i=0}^{j-1} (\norm{d_t\bH_h^{i+1}}{\Ltwo{D}}^2 +\norm{d_t\lambda_h^{i+1}}{H^{1/2}(\Gamma)}^2)+\sum_{i=0}^{j-1} 
\norm{\nabla\times (\bH^{i+1}_h-\bH^{i}_h)}{\Ltwo{D}}^2
\leq 
C.\notag
\end{align}
Clearly, if $1/2\le\theta\le1$ then~\eqref{eq:d4}
yields~\eqref{eq:denergy}. If $0\le\theta<1/2$ then since the mesh
is regular, the inverse
estimate~$\norm{\nabla\bv_h^i}{\Ltwo{D}}\lesssim
h^{-1}\norm{\bv_h^i}{\Ltwo{D}}$ gives
\begin{align*}
(2\theta-1)k^2
\sum_{i=0}^{j-1}
\norm{\nabla\bv_h^i}{\Ltwo{D}}^2
+
k
\sum_{i=0}^{j-1}
\norm{\bv_h^i}{\Ltwo{D}}^2
&\gtrsim
\left(
1-k^2h^{-1}(1-2\theta)
\right)
k
\sum_{i=0}^{j-1}
\norm{\bv_h^i}{\Ltwo{D}}^2
\\
&\gtrsim
k
\sum_{i=0}^{j-1}
\norm{\bv_h^i}{\Ltwo{D}}^2
\end{align*}
as $k^2h^{-1}\to0$ under the assumption~\eqref{eq:hk con}.
This estimate and~\eqref{eq:d4} give~\eqref{eq:denergy},
completing the proof of the lemma.
\end{proof}

Collecting the above results we obtain the following equations
satisfied by the discrete functions
%~$\bv_{hk}^-$ and~$\bH_{hk}$.
defined from~$\bm_h^i$, $\bH_h^i$, $\lambda_h^i$, and
$\bv_h^i$.

\begin{lemma}\label{lem:mhk Hhk}
Let $\bm_{hk}^{-}$, $A_{hk}^\pm:=(\bH_{hk}^\pm,\lambda_{hk}^\pm)$, and $\bv_{hk}^{-}$ be
defined from $\bm_h^i$, $\bH_h^i$, $\lambda_h^i$, and $\bv_h^i$ as
described in Subsection~\ref{subsec:dis spa}.
Then 
\begin{subequations}\label{eq:mhk Hhk}
\begin{align}
   \alpha\dual{\bv_{hk}^-}{\bphi_{hk}}_{D_T}
&+ 
\dual{(\bm_{hk}^-\times \bv_{hk}^-)}{\bphi_{hk}}_{D_T}
+ 
C_e \theta k \dual{\nabla \bv_{hk}^-}{\nabla \bphi_{hk}}_{D_T}
\nonumber
\\
&=
-C_e \dual{\nabla \bm_{hk}^-}{\nabla \bphi_{hk}}_{D_T}
+
\dual{\bH_{hk}^-}{\bphi_{hk}}_{D_T}
\label{eq:mhk Hhk1}
\\
\intertext{and with~$\partial_t$ denoting time derivative}
\int_0^T 
a_h(\partial_t A_{hk}(t),B_h)
\,dt
&+
\int_0^T 
b(\bH_{hk}^+(t),\bxi_h)
\,dt
=
-\dual{\bv_{hk}^-}{\bxi_h}_{D_T}
\label{eq:mhk Hhk2}
\end{align}
\end{subequations}
for all $\bphi_{hk}$ and $B_h:=(\bxi_h,\zeta_h)$ satisfying
$\bphi_{hk}(t,\cdot)\in\KK_{\bm_h^i}$ for $t\in[t_i,t_{i+1})$ and
$B_h\in\XX_h$.
%$(\bxi_h,\zeta_h)\in\XX_h$.
%$(\bxi_h(t,\cdot),\zeta_h(t,\cdot))\in\XX_h$ for all $t\in[0,T]$.
\end{lemma}
\begin{proof}
The lemma is a direct consequence of~\eqref{eq:dllg} 
and~\eqref{eq:eddygeneral}.
\end{proof}

The next lemma shows that the functions defined in the above
lemma form sequences which have convergent subsequences.
\begin{lemma}\label{lem:weakconv}
Assume that the assumptions~\eqref{eq:hk con} and~\eqref{eq:mh0
Hh0} hold.
As $h$, $k\to0$, the following limits exist up to extraction of
subsequences
\begin{subequations}\label{eq:weakconv}
\begin{alignat}{2}
 \bm_{hk}&\rightharpoonup\bm\quad&&\text{in }\Hone{D_T},\label{eq:wc1}\\
 \bm_{hk}^\pm&\rightharpoonup\bm\quad&&\text{in }
L^2(0,T;\Hone{D}),\label{eq:wc1a}
\\
 \bm_{hk}^\pm
& \rightarrow\bm
\quad&&
\text{in } \Ltwo{D_T}, 
\label{eq:wc2}\\
  (\bH_{hk},\lambda_{hk})&\rightharpoonup(\bH,\lambda)\quad&&\text{in }
L^2(0,T;\XX),\label{eq:wc3}\\
    (\bH_{hk}^\pm,\lambda_{hk}^\pm)&\rightharpoonup(\bH,\lambda)
\quad&&\text{in } L^2(0,T;\XX),\label{eq:wc4}\\
 (\bH_{hk},\lambda_{hk})&\rightharpoonup (\bH,\lambda) \quad&&\text{in } 
H^1(0,T;\Ltwo{D}\times H^{1/2}(\Gamma)),\label{eq:wc31}\\ 
    \bv_{hk}^- &\rightharpoonup \bm_t\quad&&\text{in }\Ltwo{D_T},\label{eq:wc5}
\end{alignat}
\end{subequations}
for certain functions $\bm$, $\bH$, and $\lambda$ satisfying
$\bm\in \Hone{D_T}$, $\bH\in 
H^1(0,T;\Ltwo{D})$, and $(\bH,\lambda)\in 
L^2(0,T;\XX)$. Here~$\rightharpoonup$ denotes the weak convergence
and~$\to$ denotes the strong convergence in the relevant space.

Moreover, if the assumption~\eqref{eq:mh0 Hh0} holds
then there holds additionally $|\bm|=1$ almost everywhere
in~$D_T$.
\end{lemma}

\begin{proof}
Note that due to the Banach-Alaoglu Theorem, to show the
existence of a weakly convergent subsequence, it suffices to show
the boundedness of the sequence in the respective norm. Thus in
order to prove~\eqref{eq:wc1} we will prove
that $\norm{\bm_{hk}}{\Hone{D_T}}\le C$ for all~$h,k>0$.

By Step~(3) of Algorithm~\ref{algorithm} and due to an idea from~\cite{bartels}, there holds for all $z\in\NN_h$
\begin{align*}
\begin{split}
 |\bm_h^j(z)|^2 &=|\bm_h^{j-1}(z)|^2+k^2|\bv_h^{j-1}(z)|^2=|\bm_h^{j-2}(z)|^2+k^2|\bv_h^{j-1}(z)|^2+k^2|\bv_h^{j-2}(z)|^2\\
 &= |\bm_h^{0}(z)|^2+k^2\sum_{i=0}^{j-1}|\bv_h^{i}(z)|^2.
 \end{split}
\end{align*}
By using the equivalence (see e.g. \cite[Lemma~3.2]{thanh})
\begin{equation}\label{eq:LT13}
\norm{\bphi}{L^p(D)}^p
\simeq
h^3 \sum_{z\in\NN_h}
|\bphi(z)|^p,
\quad 1 \le p < \infty,
\quad \bphi\in \SS^1(\TT_h)^3,
\end{equation}
we deduce that
\begin{align}\label{eq:const}
\left|
\norm{\bm_h^j}{\Ltwo{D}}^2-\norm{\bm_h^0}{\Ltwo{D}}^2
\right|
&
\simeq
h^3
\sum_{z\in\NN_h}
\left(
|\bm_h^j(z)|^2
-
|\bm_h^0(z)|^2
\right)
=
k^2\sum_{i=0}^{j-1}
h^3
\sum_{z\in\NN_h}
|\bv_h^i(z)|^2
\nonumber
\\
&
\simeq
k^2\sum_{i=0}^{j-1}\norm{\bv_h^{i}}{\Ltwo{D}}^2\leq kC_{\rm ener},
\end{align}
where in the last step we used~\eqref{eq:denergy}.
This proves immediately
 \begin{align*}
\norm{\bm_{hk}}{\Ltwo{D_T}}^2\simeq k\sum_{i=1}^N \norm{\bm^i_h}{\Ltwo{D}}^2\leq
k\sum_{i=1}^N \big(\norm{\bm^0_h}{\Ltwo{D}}^2+kC_{\rm ener}\big) 
\leq C.
 \end{align*}

On the other hand, since $\partial_t\bm_{hk} = (\bm_h^{i+1}-\bm_h^i)/k$ on
$(t_i,t_{i+1})$ for $i=0,\ldots,N-1$
and $\bm_h^{i+1}(z)-\bm_h^i(z) = k\bv_h^i(z)$ for all
$z\in\NN_h$, we have by using~\eqref{eq:denergy} and~\eqref{eq:LT13}
 \begin{align}\label{eq:dt mhk}
  \norm{\partial_t\bm_{hk}}{\Ltwo{D_T}}^2
&
=
\sum_{i=0}^{N-1}
\int_{t_j}^{t_{j+1}}
\norm{\partial_t\bm_{hk}}{\Ltwo{D}}^2 \, dt
=
k^{-1}\sum_{i=0}^{N-1} \norm{\bm^{i+1}_h-\bm^i_h}{\Ltwo{D}}^2
\nonumber
\\
&\simeq 
k^{-1}
\sum_{i=0}^{N-1} 
h^3
\sum_{z\in \NN_h} |\bm^{i+1}_h(z)-\bm^i_h(z)|^2
=
k
\sum_{i=0}^{N-1} 
h^3
\sum_{z\in \NN_h} |\bv^i_h(z)|^2
\nonumber
\\
&\simeq
k
\sum_{i=0}^{N-1} 
\norm{\bv_h^i}{\Ltwo{D}}^2
\leq
C_{\rm ener}.
\end{align}
%By definition, there holds $\partial_t\bm_h^i(z)=(\bm^{i+1}_h(z)-\bm^i_h(z))/k= \bv_h^i(z)$ and hence
%\begin{align*}
%  \norm{\partial_t\bm_{hk}}{L^2(D_T)}^2\lesssim k\sum_{i=1}^{N-1} \sum_{T\in\TT_h}|T| \sum_{z\in \NN_h\cap T}|\bv_h^i(z)|^2\simeq k\sum_{i=1}^{N-1}\norm{\bv^i_h}{\Ltwo{D}}^2\leq C_{\rm ener}.
%\end{align*}
Finally the gradient $\nabla\bm_{hk}$
is shown to be bounded by using~\eqref{eq:denergy}
again as follows:
\begin{align*}
\norm{\nabla\bm_{hk}}{\Ltwo{D_T}}^2\simeq k\sum_{i=1}^N
\norm{\nabla\bm^i_h}{\Ltwo{D}}^2\leq C_{\rm ener}kN\leq C_{\rm
energy}T.
 \end{align*}
 Altogether, we showed that $\{\bm_{hk}\}$
is a bounded sequence in $\Hone{D_T}$ and thus posesses a weakly convergent
subsequence, i.e., we proved~\eqref{eq:wc1}. 

In particular,~\eqref{eq:mh0 Hh0}, \eqref{eq:denergy},
and~\eqref{eq:const} imply
\begin{equation}\label{eq:mhk pm}
\norm{\bm_{hk}^\pm}{L^2(0,T;\Hone{D})}
\le
\norm{\bm_{hk}^\pm}{L^\infty(0,T;\Hone{D})}
\lesssim
C_{\rm ener},
\end{equation}
yielding~\eqref{eq:wc1a}.

We prove~\eqref{eq:wc2} for $\bm_{hk}^-$ only; similar arguments
hold for $\bm_{hk}^+$. First, we note that the definition
of~$\bm_{hk}$ and~$\bm_{hk}^-$, and the estimate~\eqref{eq:dt mhk}
imply, for all $t\in[t_j,t_{j+1})$,
\begin{align*}
\norm{\bm_{hk}(t,\cdot)-\bm_{hk}^-(t,\cdot)}{\Ltwo{D}}
&=
\norm{(t-t_j)\frac{\bm_{h}^{j+1}-\bm_{h}^j}{k}}{\Ltwo{D}}
%\\
%&
\le
k\norm{\partial_t\bm_{hk}(t,\cdot)}{\Ltwo{D}}
\lesssim
k C_{\rm ener}.
\end{align*}
This in turn implies
\[
\norm{\bm_{hk}-\bm_{hk}^-}{\Ltwo{D_T}}
\lesssim
kT C_{\rm ener}
\to 0 \quad\text{as }h,k\to0.
\]
Thus,~\eqref{eq:wc2} follows from the triangle
inequality,~\eqref{eq:wc1}, and the Sobolev embedding.

Statement~\eqref{eq:wc3} follows immediately
from~\eqref{eq:denergy} by noting that
 \begin{align*}
  \norm{(\bH_{hk},\lambda_{hk})}{L^2(0,T;\XX)}^2
\simeq 
k\sum_{i=1}^N 
\big(
\norm{\bH_h^i}{\Hcurl{D}}^2+\norm{\lambda_{h}^i}{H^{1/2}(\Gamma)}^2
\big)
\leq 
kNC_{\rm ener}\leq TC_{\rm ener}.
 \end{align*}
%and extraction of a weakly convergent subsequence. 
The proof of~\eqref{eq:wc4} follows analogously. Consequently, we
obtain~\eqref{eq:wc31} by using again~\eqref{eq:denergy} and the
above estimate as follows:
\begin{align*}
  \norm{\bH_{hk}}{H^1(0,T;\Ltwo{D})}^2
&\simeq 
\norm{\bH_{hk}}{\Ltwo{D_T}}^2 + k\sum_{i=1}^N\norm{d_t\bH_h^i}{\Ltwo{D}}^2
%\\
%&
\leq TC_{\rm ener}+C_{\rm ener}.
\end{align*}
The convergence of $\lambda_{hk}$ in the statement follows analogously.
Finally,~\eqref{eq:wc5} follows from $\partial_t\bm_{hk}(t)=\bv^-_{hk}(t)$ and~\eqref{eq:wc1}. 

To show that $\bm$ satisfies the constraint $|\bm|=1$,
we first note that
 \begin{align*}
  \norm{|\bm|-1}{L^2(D_T)}\leq \norm{\bm-\bm_{hk}}{\Ltwo{D_T}}+\norm{|\bm_{hk}|-1}{L^2(D_T)}.
 \end{align*}
The first term on the right-hand side
converges to zero due to~\eqref{eq:wc1} and the
compact embedding of~$\Hone{D_T}$ in~$\Ltwo{D_T}$. 
For the second term, we note that
\begin{align*}
\norm{1-|\bm_{hk}|}{L^2(D_T)}^2
&\lesssim
k\sum_{j=0}^N\big(\norm{|\bm_{h}^j|-|\bm_h^0|}{L^2(D)}^2
+
\norm{1-|\bm_h^0|}{L^2(D)}^2\big)
\nonumber\\
&\leq  k\sum_{j=0}^N\big(\norm{|\bm_{h}^j|^2
-
|\bm_h^0|^2}{L^1(D)}+\norm{|\bm^0|-|\bm_h^0|}{L^2(D)}^2\big)
 %\end{split}
\end{align*}
where we used $(x-y)^2\leq |x^2-y^2|$ for all $x,y\ge0$.
Similarly to~\eqref{eq:const} it can be shown that
\begin{equation}\label{eq:mhj mh0}
\norm{|\bm_h^j|^2-|\bm_h^0|^2}{L^1(D)}
\simeq
k^2\sum_{i=0}^{j-1}
\norm{\bv_h^{i}}{\Ltwo{D}}^2
\leq kC_{\rm ener}.
\end{equation}
Hence
\[
\norm{1-|\bm_{hk}|}{L^2(D_T)}^2
\leq kC_{\rm ener} +\norm{\bm^0-\bm_h^0}{\Ltwo{D_T}}^2\to
0\quad\text{as }h,k\to 0.
\]
Altogether, we showed $|\bm|=1$ almost everywhere in~$D_T$,
completing the proof of the lemma.
\end{proof}

We also need the following strong convergence property.
\begin{lemma}\label{lem:str con}
Under the assumptions~\eqref{eq:hk 12} and~\eqref{eq:mh0 Hh0} there holds
\begin{equation}\label{eq:mhk h12}
\norm{\bm_{hk}^--\bm}{L^2(0,T;\H^{1/2}(D))}
\to 0
\quad\text{as } h,k\to 0.
\end{equation}
\end{lemma}
\begin{proof}
It follows from the triangle inequality and the definitions of~$\bm_{hk}$
and~$\bm_{hk}^-$ that
\begin{align*}
 \norm{\bm_{hk}^--\bm}{L^2(0,T;\H^{1/2}(D))}^2
&\lesssim 
\norm{\bm_{hk}^--\bm_{hk}}{L^2(0,T;\H^{1/2}(D))}^2
+
\norm{\bm_{hk}-\bm}{L^2(0,T;\H^{1/2}(D))}^2 \\
&\leq
\sum_{i=0}^{N-1}
k^3\norm{\bv_h^i}{\H^{1/2}(D)}^2
+
\norm{\bm_{hk}-\bm}{L^2(0,T;\H^{1/2}(D))}^2 \\
&\leq 
\sum_{i=0}^{N-1}k^3\norm{\bv_h^i}{\Hone{D}}^2
+
\norm{\bm_{hk}-\bm}{L^2(0,T;\H^{1/2}(D))}^2.
\end{align*}
The second term on the right-hand side converges to zero due
to~\eqref{eq:wc1} and the compact embedding of
\[
\Hone{D_T}
\simeq
\{\bv \, | \,
\bv\in L^2(0,T;\Hone{D}), \, \bv_t\in L^2(0,T;\Ltwo{D}) \}
\]
into $L^2(0,T;\H^{1/2}(D))$; see \cite[Theorem 5.1]{Lio69}.
For the first term on the right-hand side,
when $\theta>1/2$,~\eqref{eq:denergy} implies
$
\sum_{i=0}^{N-1}k^3\norm{\bv_h^i}{\Hone{D}}^2
\lesssim k \to 0.
$
When $0\le\theta\leq 1/2$, a standard inverse inequality,~\eqref{eq:denergy}
and~\eqref{eq:hk 12} yield
\[
\sum_{i=0}^{N-1}k^3\norm{\bv_h^i}{\Hone{D}}^2
\lesssim 
\sum_{i=0}^{N-1}h^{-2}k^3\norm{\bv_h^i}{\Ltwo{D}}^2
\lesssim 
h^{-2}k^2 \to 0,
\]
completing the proof of the lemma.
\end{proof}

The following lemma involving the $\L^2$-norm of the cross product
of two vector-valued functions will be used when passing
to the limit of equation~\eqref{eq:mhk Hhk1}.
\begin{lemma}\label{lem:h}
There exists a constant $C_{\rm sob}>0$ which depends only on
$D$ such that
\begin{equation}
\label{eq:l2}
  \norm{\bw_{0}\times\bw_1}{\Ltwo{D}}
\leq 
C_{\rm sob}
\norm{\bw_{0}}{\H^{1/2}(D)}\norm{\bw_1}{\Hone{D}}.
 \end{equation}
 for all $\bw_0\in\H^{1/2}(D)$ and $\bw_{1}\in\Hone{D}$.
 \end{lemma}
\begin{proof}
It is shown in~\cite[Theorem~5.4, Part~I]{adams} that
the embedding $\iota\colon\Hone{D}\to\L^6(D)$ is continuous.
Obviously, the identity $\iota\colon \Ltwo{D}\to\Ltwo{D}$ is
continous. By real interpolation, we find that $\iota\colon
[\Ltwo{D},\Hone{D}]_{1/2}\to [\Ltwo{D},\L^6(D)]_{1/2}$ is
continuous. Well-known results in interpolation theory show
$
[\Ltwo{D},\Hone{D}]_{1/2}= \H^{1/2}(D)
$
and
$
[\Ltwo{D},\L^6(D)]_{1/2}=\L^3(D)
$
with equivalent norms; see e.g.~\cite[Theorem~5.2.1]{BL}.
By using H\"older's inequality, we deduce
\begin{align*}
 \norm{\bw_{0}\times\bw_{1}}{\Ltwo{D}}\leq
\norm{\bw_{0}}{\L^3(D)}\norm{\bw_{1}}{\L^6(D)}
\lesssim
\norm{\bw_{0}}{\H^{1/2}(D)}\norm{\bw_1}{\Hone{D}},
\end{align*}
proving the lemma.
\end{proof}

Finally, to pass to the limit in equation~\eqref{eq:mhk Hhk2} we
need the following result.
\begin{lemma}\label{lem:wea con}
For any sequence~$\{\lambda_h\}\subset H^{1/2}(\Gamma)$ 
and any function~$\lambda\in H^{1/2}(\Gamma)$, if
\begin{equation}\label{eq:zet con}
\lim_{h\to0}
\dual{\lambda_h}{\nu}_{\Gamma}
=
\dual{\lambda}{\nu}_{\Gamma}
\quad\forall\nu\in H^{-1/2}(\Gamma)
\end{equation}
then
\begin{equation}\label{eq:dtn zet con}
\lim_{h\to0}
\dual{\dtn_h\lambda_h}{\zeta}_{\Gamma}
=
\dual{\dtn\lambda}{\zeta}_{\Gamma}
\quad\forall\zeta\in H^{1/2}(\Gamma).
\end{equation}
\end{lemma}
\begin{proof}
Let~$\mu$ and~$\mu_h$ be defined
by~\eqref{eq:bem} with~$\lambda$ in the second equation replaced 
by~$\lambda_h$. Then (recalling that Costabel's symmetric
coupling is used) $\dtn\lambda$ and~$\dtn_h\lambda_h$ are
defined via~$\mu$ and~$\mu_h$ by~\eqref{eq:dtn2} 
and~\eqref{eq:bem2}, respectively, namely,
$\dtn\lambda = (1/2-\dlp^\prime)\mu - \hyp\lambda$ and
$
\dual{\dtn_h\lambda_h}{\zeta_h}_\Gamma 
= 
\dual{(1/2-\dlp^\prime)\mu_h}{\zeta_h}_\Gamma 
-
\dual{\hyp \lambda_h}{\zeta_h}_\Gamma
$
for all~$\zeta_h\in \SS^1(\TT_h|_\Gamma)$.
For any~$\zeta\in H^{1/2}(\Gamma)$, 
let~$\{\zeta_h\}$ be a sequence
in~$\SS^1(\TT_h|_\Gamma)$ 
satisfying~$\lim_{h\to0}\norm{\zeta_h-\zeta}{H^{1/2}(\Gamma)}=0$.
By using the triangle inequality and the above representations
of~$\dtn\lambda$ and~$\dtn_h\lambda_h$ we deduce
\begin{align}\label{eq:dtn dtn}
\big|
\dual{\dtn_h\lambda_h}{\zeta}
-
\dual{\dtn\lambda}{\zeta}_\Gamma
\big| 
&\leq 
\big|\dual{\dtn_h\lambda_h-\dtn\lambda}{\zeta_h}_\Gamma\big|
+
\big|\dual{\dtn_h\lambda_h-\dtn\lambda}{\zeta-\zeta_h}_\Gamma\big|
\notag
\\ 
&\le
\big|
\dual{(\tfrac12-\dlp^\prime)(\mu_h-\mu)}{\zeta_h}_\Gamma 
\big|
+
\big|
\dual{\hyp (\lambda_h-\lambda)}{\zeta_h}_\Gamma
\big|
\notag
\\
&\quad
+
\big|\dual{\dtn_h\lambda_h-\dtn\lambda}{\zeta-\zeta_h}_\Gamma\big|
\notag
\\
&\le
\big|
\dual{(\tfrac12-\dlp^\prime)(\mu_h-\mu)}{\zeta_h}_\Gamma 
\big|
+
\big|
\dual{\hyp (\lambda_h-\lambda)}{\zeta}_\Gamma
\big|
\notag
\\
&\quad
+
\big|
\dual{\hyp (\lambda_h-\lambda)}{\zeta_h-\zeta}_\Gamma
\big|
+
\big|\dual{\dtn_h\lambda_h-\dtn\lambda}{\zeta-\zeta_h}_\Gamma\big|.
\end{align}
The second term on the right-hand side of~\eqref{eq:dtn dtn}
goes to zero as~$h\to0$ due to~\eqref{eq:zet con} and the
self-adjointness of~$\hyp$. The third term converges to zero
due to the strong convergence~$\zeta_h\to\zeta$
in~$H^{1/2}(\Gamma)$ and the boundedness of~$\{\lambda_h\}$
in~$H^{1/2}(\Gamma)$, which is a consequence of~\eqref{eq:zet
con} and the Banach-Steinhaus Theorem. The last term tends to
zero due to the convergence of~$\{\zeta_h\}$ and the boundedness
of~$\{\dtn_h\lambda_h\}$; see~\eqref{eq:dtnelliptic}.
Hence~\eqref{eq:dtn zet con} is proved if we prove
\begin{equation}\label{eq:muh mu}
\lim_{h\to0}
\dual{(1/2-\dlp^\prime)(\mu_h-\mu)}{\zeta_h}_{\Gamma}
= 0.
\end{equation}
We have
\begin{align}\label{eq:muh mu3}
\dual{(\tfrac12-\dlp^\prime)(\mu_h-\mu)}{\zeta_h}_\Gamma 
=
\dual{\mu_h-\mu}{(\tfrac12-\dlp)\zeta}_\Gamma 
+
\dual{\mu_h-\mu}{(\tfrac12-\dlp)(\zeta_h-\zeta)}_\Gamma. 
\end{align}
The definition of~$\mu_h$ implies
$
\norm{\mu_h}{H^{-1/2}(\Gamma)}
\lesssim
\norm{\lambda_h}{H^{1/2}(\Gamma)}
\lesssim
1,
$
and therefore the second term on the right-hand side
of~\eqref{eq:muh mu3} goes to zero.
Hence it suffices to prove
\begin{equation}\label{eq:muh mu2}
\lim_{h\to0}
\dual{\mu_h-\mu}{\eta}_\Gamma
= 0
\quad\forall\eta\in H^{1/2}(\Gamma).
\end{equation}
Since~$\slp : H^{-1/2}(\Gamma) \to
H^{1/2}(\Gamma)$ is bijective and self-adjoint, 
for any~$\eta\in H^{1/2}(\Gamma)$
there exists~$\nu\in H^{-1/2}(\Gamma)$ such that
\[
\dual{\mu_h-\mu}{\eta}_{\Gamma}
=
\dual{\mu_h-\mu}{\slp\nu}_{\Gamma}
=
\dual{\slp(\mu_h-\mu)}{\nu}_{\Gamma}
=
\dual{\slp(\mu_h-\mu)}{\nu_h}_{\Gamma}
+
\dual{\slp(\mu_h-\mu)}{\nu-\nu_h}_{\Gamma},
\]
where~$\{\nu_h\}\subset\PP^0(\TT_h|\Gamma)$ is a sequence
satisfying~$\norm{\nu_h-\nu}{H^{-1/2}(\Gamma)}\to0$. 
The definitions of~$\mu_h$ and~$\mu$, and the above equation imply
\begin{align*}
\dual{\mu_h-\mu}{\eta}_{\Gamma}
&=
\dual{(\dlp-\tfrac12)(\lambda_h-\lambda)}{\nu_h}_{\Gamma}
+
\dual{\slp(\mu_h-\mu)}{\nu-\nu_h}_{\Gamma}
\\
&=
\dual{\lambda_h-\lambda}{(\dlp^\prime-\tfrac12)\nu_h}_{\Gamma}
+
\dual{\slp(\mu_h-\mu)}{\nu-\nu_h}_{\Gamma}
\\
&=
\dual{\lambda_h-\lambda}{(\dlp^\prime-\tfrac12)\nu}_{\Gamma}
+
\dual{\lambda_h-\lambda}{(\dlp^\prime-\tfrac12)(\nu_h-\nu)}_{\Gamma}
%\\
%&\quad
+
\dual{\slp(\mu_h-\mu)}{\nu-\nu_h}_{\Gamma}.
\end{align*}
The first two terms on the right-hand side go to zero due to the
convergence of~$\{\lambda_h\}$ and~$\{\nu_h\}$. The last term
also approaches zero if we note the boundedness of~$\{\mu_h\}$.
This proves~\eqref{eq:muh mu2} and completes the proof of the lemma.
\end{proof}

\subsection{Proof of Theorem~\ref{thm:weakconv}}\label{section:weak}
We are now ready to prove that the
problem~\eqref{eq:strong}--\eqref{eq:con} has a weak solution.
\begin{proof}
We recall from~\eqref{eq:wc1}--\eqref{eq:wc5} that $\bm\in \Hone{D_T}$, 
$(\bH,\lambda)\in L^2(0,T;\XX)$ and $\bH\in H^1(0,T;\Ltwo{D})$.
By virtue of Lemma~\ref{lem:bil for} it suffices to prove that
$(\bm,\bH,\lambda)$ satisfies~\eqref{eq:wssymm1} 
and~\eqref{eq:bil for}.

Let $\bphi\in C^\infty(D_T)$ and $B:=(\bxi,\zeta)\in L^2(0,T;\XX)$.
On the one hand, we define the test function
$\bphi_{hk}:=\Pi_{\SS}(\bm_{hk}^-\times\bphi)$
as the usual interpolant of~$\bm_{hk}^-\times\bphi$ into
$\SS^1(\TT_h)^3$. 
By definition, $\bphi_{hk}(t,\cdot) \in\KK_{\bm_h^j}$ for all
$t\in[t_j,t_{j+1})$. 
On the other hand, it follows from Lemma~\ref{lem:den pro}
that there exists~$B_h:=(\bxi_h,\zeta_h)\in\XX_h$ converging
to~$B\in\XX$.
Equations~\eqref{eq:mhk Hhk} hold with these
test functions. The main idea of the proof is to
pass to the limit in~\eqref{eq:mhk Hhk1} and~\eqref{eq:mhk Hhk2} to
obtain~\eqref{eq:wssymm1} and~\eqref{eq:bil for}, respectively.

In order to prove that~\eqref{eq:mhk Hhk1}
implies~\eqref{eq:wssymm1} we will prove that as~$h,k\to0$
\begin{subequations}\label{eq:conv}
\begin{align}
\dual{\bv_{hk}^-}{\bphi_{hk}}_{D_T} 
&\to
\dual{\bm_t}{ \bm\times\bphi}_{D_T},
\label{eq:conv1}
\\
\dual{\bm_{hk}^-\times\bv_{hk}^-}{\bphi_{hk}}_{D_T}
&\to
\dual{\bm\times\bm_t}{\bm\times\bphi}_{D_T},
\label{eq:conv2}
\\
k\dual{\nabla\bv_{hk}^-}{\nabla\bphi_{hk}}_{D_T}
&\to0,
\label{eq:conv3}
\\
\dual{\nabla\bm_{hk}^-}{\nabla\bphi_{hk}}_{D_T} 
&\to
\dual{\nabla\bm}{\nabla(\bm\times\bphi)}_{D_T},
\label{eq:conv4}
\\
\dual{\bH_{hk}^-}{\bphi_{hk}}_{D_T} 
&\to
\dual{\bH}{ \bm\times\bphi}_{D_T}.
\label{eq:conv5}
\end{align}
\end{subequations}

Firstly, it can be easily shown that (see~\cite{alouges})
\begin{equation}\label{eq:phi mhk}
\norm{\bphi_{hk}-\bm_{hk}^-\times\bphi}{L^2(0,T;\Hone{D})}
\lesssim
h
\norm{\bm_{hk}^-}{L^2(0,T;\Hone{D})}
\norm{\bphi}{{\mathbb W}^{2,\infty}(D_T)}
\lesssim
h
\norm{\bphi}{{\mathbb W}^{2,\infty}(D_T)}
\end{equation}
and
\begin{align}\label{eq:phiinfty}
 \norm{\bphi_{hk}-\bm_{hk}^-\times\bphi}{L^\infty(0,T;\Hone{D})}
\lesssim
h
\norm{\bm_{hk}^-}{L^\infty(0,T;\Hone{D})}
\norm{\bphi}{{\mathbb W}^{2,\infty}(D_T)}
\lesssim
h
\norm{\bphi}{{\mathbb W}^{2,\infty}(D_T)},
\end{align}
where we used~\eqref{eq:mhk pm}.
In particular, we have
\begin{equation}\label{eq:phi hk inf}
\norm{\bphi_{hk}}{L^\infty(0,T;\Hone{D})}
\lesssim
1.
\end{equation}
We now prove~\eqref{eq:conv1} and~\eqref{eq:conv5}.
With~\eqref{eq:phi mhk}, there holds for $h,k\to0$,
\begin{align}\label{eq:phihk mhk}
\norm{\bphi_{hk}-\bm\times\bphi}{\Ltwo{D_T}}
&\le
\norm{\bphi_{hk}-\bm_{hk}^-\times\bphi}{\Ltwo{D_T}}
+
\norm{(\bm_{hk}^--\bm)\times\bphi}{\Ltwo{D_T}}
\nonumber
\\
&\lesssim
\big(h + \norm{\bm_{hk}^--\bm}{\Ltwo{D_T}}\big)
\norm{\bphi}{{\mathbb W}^{2,\infty}(D_T)}
\to0 
\end{align}
due to~\eqref{eq:wc2}. Consequently, with the
help of~\eqref{eq:wc31} and~\eqref{eq:wc5} we
obtain~\eqref{eq:conv1} and~\eqref{eq:conv5}.

In order to prove~\eqref{eq:conv2} we note that
the elementary identity
\begin{equation}\label{eq:abc}
\ba\cdot(\bb\times\bc)= \bb\cdot(\bc\times\ba)=
\bc\cdot(\ba\times\bb)
\quad\forall\ba,\bb,\bc\in\R^3
\end{equation}
yields
\begin{align}\label{eq:idid}
\dual{\bm_{hk}^-\times\bv_{hk}^-}{\bphi_{hk}}_{D_T}
=
\dual{\bv_{hk}^-}{\bphi_{hk}\times\bm_{hk}^-}_{D_T}.
\end{align}
It follows successively from the triangle inequality,
~\eqref{eq:l2} and~\eqref{eq:phi hk inf} that
\begin{align*}
 \norm{\bphi_{hk}&\times\bm_{hk}^{-}
-
(\bm\times\bphi)\times\bm}{\Ltwo{D_T}}\\
 &\leq
 \norm{\bphi_{hk}\times(\bm_{hk}^--\bm)}{\Ltwo{D_T}}+
 \norm{(\bphi_{hk}-(\bm\times\bphi))\times\bm}{\Ltwo{D_T}}\\
 &\lesssim
\Big(\int_0^T\norm{\bphi_{hk}(t)}{\Hone{D}}^2\norm{\bm_{hk}^-(t)-\bm(t)}{\H^{1/2}(D)}^2\,dt\Big)^{1/2}+
\norm{\bphi_{hk}-(\bm\times\bphi)}{\Ltwo{D_T}}\\
&\leq 
\norm{\bphi_{hk}}{L^\infty(0,T;\Hone{D})}
\norm{\bm_{hk}^--\bm}{L^2(0,T;\H^{1/2}(D))}
+
\norm{\bphi_{hk}-(\bm\times\bphi)}{\Ltwo{D_T}} \\
&\lesssim
\norm{\bm_{hk}^--\bm}{L^2(0,T;\H^{1/2}(D))}
+
\norm{\bphi_{hk}-(\bm\times\bphi)}{\Ltwo{D_T}}.
\end{align*}
Thus~\eqref{eq:mhk h12} and~\eqref{eq:phihk mhk} imply
$\bphi_{hk}\times\bm_{hk}^-\to
(\bm\times\bphi)\times\bm$ in $\Ltwo{D_T}$.
This together with~\eqref{eq:wc5} and~\eqref{eq:idid} implies
\[
\dual{\bm_{hk}^-\times\bv_{hk}^-}{\bphi_{hk}}_{D_T}
\to
\dual{\bm_t}{(\bm\times\bphi)\times\bm}_{D_T},
\]
which is indeed~\eqref{eq:conv2} by invoking~\eqref{eq:abc}.

Statement~\eqref{eq:conv4} follows from~\eqref{eq:phi mhk}, 
\eqref{eq:wc1a}, and~\eqref{eq:wc2} as follows:
As $h,k\to 0$,
\begin{align*}
\dual{\nabla\bm_{hk}^-}{\nabla\bphi_{hk}}_{D_T} 
&=
\dual{\nabla\bm_{hk}^-}{\nabla(\bphi_{hk}-\bm_{hk}^-\times\bphi)}_{D_T} 
+
\dual{\nabla\bm_{hk}^-}{\nabla(\bm_{hk}^-\times\bphi)}_{D_T} 
\\
&=
\dual{\nabla\bm_{hk}^-}{\nabla(\bphi_{hk}-\bm_{hk}^-\times\bphi)}_{D_T} 
+
\dual{\nabla\bm_{hk}^-}{\bm_{hk}^-\times\nabla\bphi}_{D_T} 
\\
&\longrightarrow
\dual{\nabla\bm}{0}_{D_T} +
\dual{\nabla\bm}{\bm\times\nabla\bphi}_{D_T} 
=
\dual{\nabla\bm}{\nabla(\bm\times\bphi)}_{D_T} .
\end{align*}
%as $h,k\to 0$.

Finally, in order to prove~\eqref{eq:conv3} we first note that 
\eqref{eq:phi mhk} and the boundedness 
of the sequence~$\{\norm{\bm_{hk}^-}{L^2(0,T;\Hone{D})}\}$,
see~\eqref{eq:mhk pm},
give the boundedness
of~$\{\norm{\bphi_{hk}}{L^2(0,T;\Hone{D}}\}$, and thus
of~$\{\norm{\nabla\bphi_{hk}}{\Ltwo{D_T}}\}$. 
On the other hand, 
\begin{equation}\label{eq:vhk}
\norm{\nabla\bv_{hk}^-}{\Ltwo{D_T}}^2
=
k
\sum_{i=0}^{N-1}
\norm{\nabla\bv_{h}^i}{\Ltwo{D}}^2.
\end{equation}
If $1/2<\theta\le1$ then~\eqref{eq:denergy} and~\eqref{eq:vhk}
yield the boundedness of $\{\norm{\nabla\bv_{hk}^-}{\Ltwo{D_T}}\}$.
Hence
\[
k\dual{\nabla\bv_{hk}^-}{\nabla\bphi_{hk}}_{D_T}
\to0
\quad\text{as }h,k\to 0.
\]
If $0\le\theta\le1/2$ then the inverse estimate,~\eqref{eq:vhk},
and~\eqref{eq:denergy} yield
\[
\norm{\nabla\bv_{hk}^-}{\Ltwo{D_T}}^2
\lesssim
kh^{-2}
\sum_{i=0}^{N-1}
\norm{\bv_{h}^i}{\Ltwo{D}}^2
\lesssim h^{-2},
\]
so that
$
\left|
k\dual{\nabla\bv_{hk}^-}{\nabla\bphi_{hk}}_{D_T}
\right|
\lesssim
kh^{-1}.
$
This goes to 0 under the assumption~~\eqref{eq:hk 12}.
Altogether, we obtain~\eqref{eq:wssymm1} when passing to the limit
in~\eqref{eq:mhk Hhk1}.

Next, recalling that~$B_h\to B$ in~$\XX$
we prove that~\eqref{eq:mhk Hhk2}
implies~\eqref{eq:bil for} by proving
\begin{subequations}\label{eq:convb}
\begin{align}
\dual{\partial_t\bH_{hk}}{\bxi_h}_{D_T}
&\to
\dual{\bH_t}{\bxi}_{D_T},
\label{eq:convb1} 
\\
\dual{\dtn_h\partial_t\lambda_{hk}}{\zeta_h}_{\Gamma_T}
&\to
\dual{\dtn\lambda_t}{\zeta}_{\Gamma_T}, \label{eq:spec}
\\
\dual{\nabla\times\bH_{hk}^+}{\nabla\times\bxi_h}_{D_T}
&\to
\dual{\nabla\times\bH}{\nabla\times\bxi}_{D_T},
\\
\dual{\bv_{hk}^-}{\bxi_h}_{D_T}
&\to
\dual{\bv}{\bxi}_{D_T}.
\end{align}
\end{subequations}
The proof is similar to that of~\eqref{eq:conv} (where we
use Lemma~\ref{lem:wea con}
for the proof of~\eqref{eq:spec}) and is
therefore omitted. This proves~(3) and~(5) of
Definition~\ref{def:fembemllg}.

Finally, we obtain $\bm(0,\cdot)=\bm^0$, $\bH(0,\cdot)=\bH^0$, and
$\lambda(0,\cdot)=\lambda^0$ from the weak convergence and the
continuity of the trace operator.
This and $|\bm|=1$ yield Statements~(1)--(2) of
Definition~\ref{def:fembemllg}. To obtain~(4), note that
$\nabla_\Gamma\colon H^{1/2}(\Gamma)\to \H_\perp^{-1/2}(\Gamma)$ and
$\bn\times(\bn\times(\cdot))\colon \Hcurl{D}\to  \H_\perp^{-1/2}(\Gamma)$ are bounded
linear operators; see~\cite[Section~4.2]{buffa2} for exact definition of the spaces and the result. 
Weak convergence then proves~(4) of
Definition~\ref{def:fembemllg}.
Estimate~\eqref{eq:energybound2} follows by weak
lower-semicontinuity and the energy bound~\eqref{eq:denergy}.
This completes the proof of the theorem.
\end{proof}

\section{Numerical experiment}\label{section:numerics}
The following numerical experiment is carried out by use of the
FEM toolbox FEniCS~\cite{fenics} (\texttt{fenicsproject.org})
and the BEM toolbox BEM++~\cite{bempp} (\texttt{bempp.org}).
We use GMRES to solve the linear systems and blockwise diagonal
scaling as preconditioners.

The values of the constants in this example are taken from the
standard problem \#1 proposed by the Micromagnetic Modelling Activity
Group at the National Institute of Standards and
Technology~\cite{mumag}.
As domain serves the unit cube $D=[0,1]^3$ with initial conditions
\begin{align*}
 \bm^0(x_1,x_2,x_3):=\begin{cases} (0,0,-1)&\text{for } d(x)\geq 1/4,\\
                      (2Ax_1,2Ax_2,A^2-d(x))/(A^2+d(x))&\text{for }d(x)<1/4,
                     \end{cases}
\end{align*}
where $d(x):= |x_1-0.5|^2+|x_2-0.5|^2$ and $A:=(1-2\sqrt{d(x)})^4/4$ and
\begin{align*}
 \bH^0= \begin{cases} (0,0,2)&\text{in } D,\\
         (0,0,2)-\bm^0&\text{in } D^\ast.
        \end{cases}
\end{align*}
We choose the constants
\begin{align*}
 \alpha=0.5,\quad \sigma=\begin{cases}1&\text{in }D,\\ 0& \text{in }D^\ast,\end{cases}\quad \mu_0=1.25667\times 10^{-6},\quad C_e=\frac{2.6\times 10^{-11}}{\mu_0 \,6.4\times 10^{11}}.
\end{align*}
For time and space discretisation of $D_T:= [0,5]\times D$, we apply a
uniform partition in space ($h=0.1$) and time ($k=0.002$).
Figure~\ref{fig:en} plots the corresponding energies over time.
Figure~\ref{fig:m} shows a series of magnetizations $\bm(t_i)$ at certain times $t_i\in[0,5]$. 
Figure~\ref{fig:h} shows that same for the magnetic field $\bH(t_i)$.
\begin{figure}
\psfrag{energy}{\tiny energy}
\psfrag{time}{\tiny time $t$}
\psfrag{menergy}{\tiny magnetization energy}
\psfrag{henergy}{\tiny magnetic energy}
\psfrag{sum}{\tiny total energy}
\includegraphics[width=0.6\textwidth]{pics/energy2.eps}
\caption{The magnetization engergy
$\norm{\nabla\bm_{hk}(t)}{\Ltwo{D}}$ and the energy of the magnetic
field $\norm{\bH_{hk}(t)}{\Hcurl{D}}$ plotted over the time.}
\label{fig:en}
\end{figure}

\begin{figure}
\includegraphics[width=0.24\textwidth]{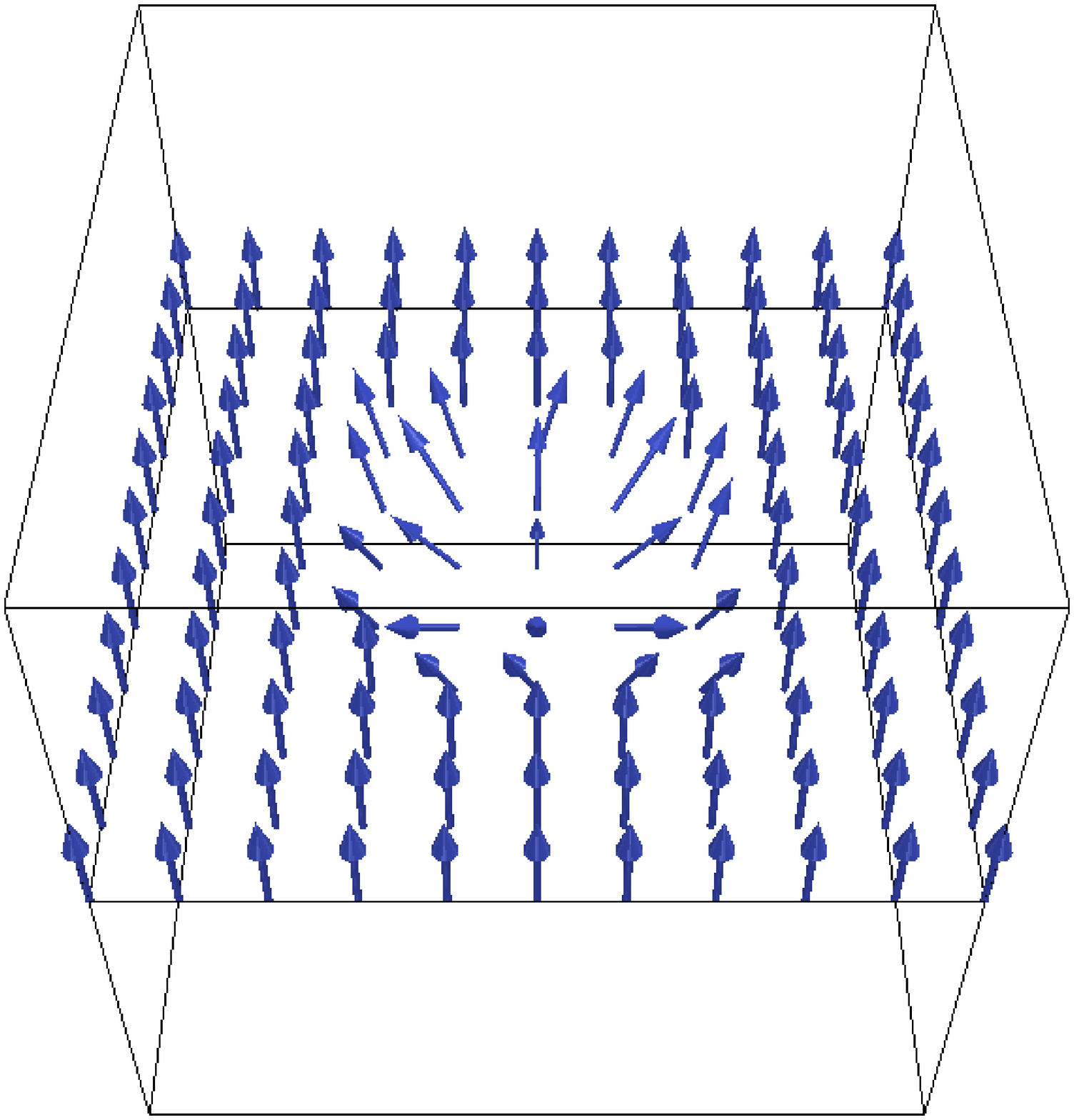}
\includegraphics[width=0.24\textwidth]{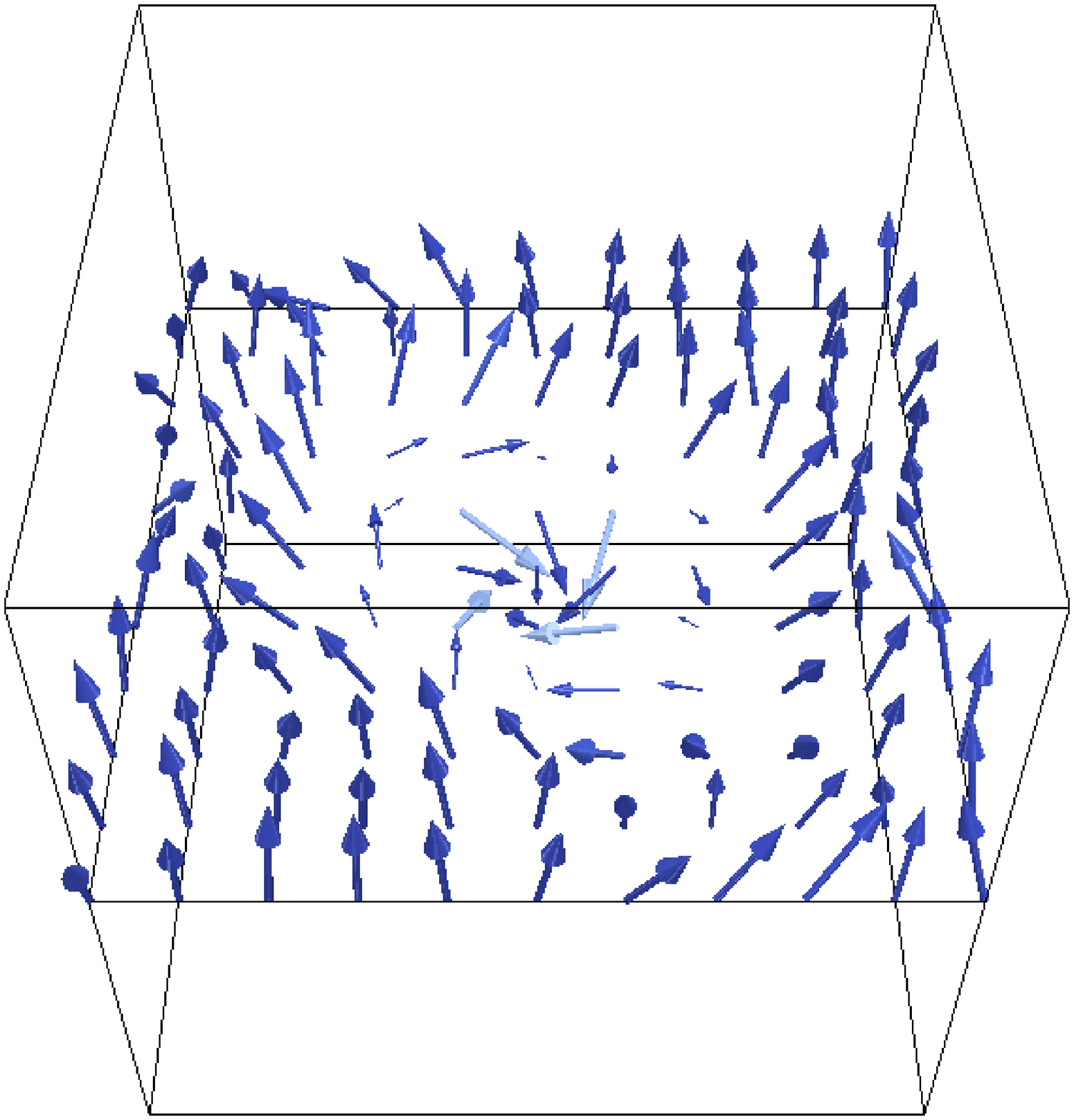}
\includegraphics[width=0.24\textwidth]{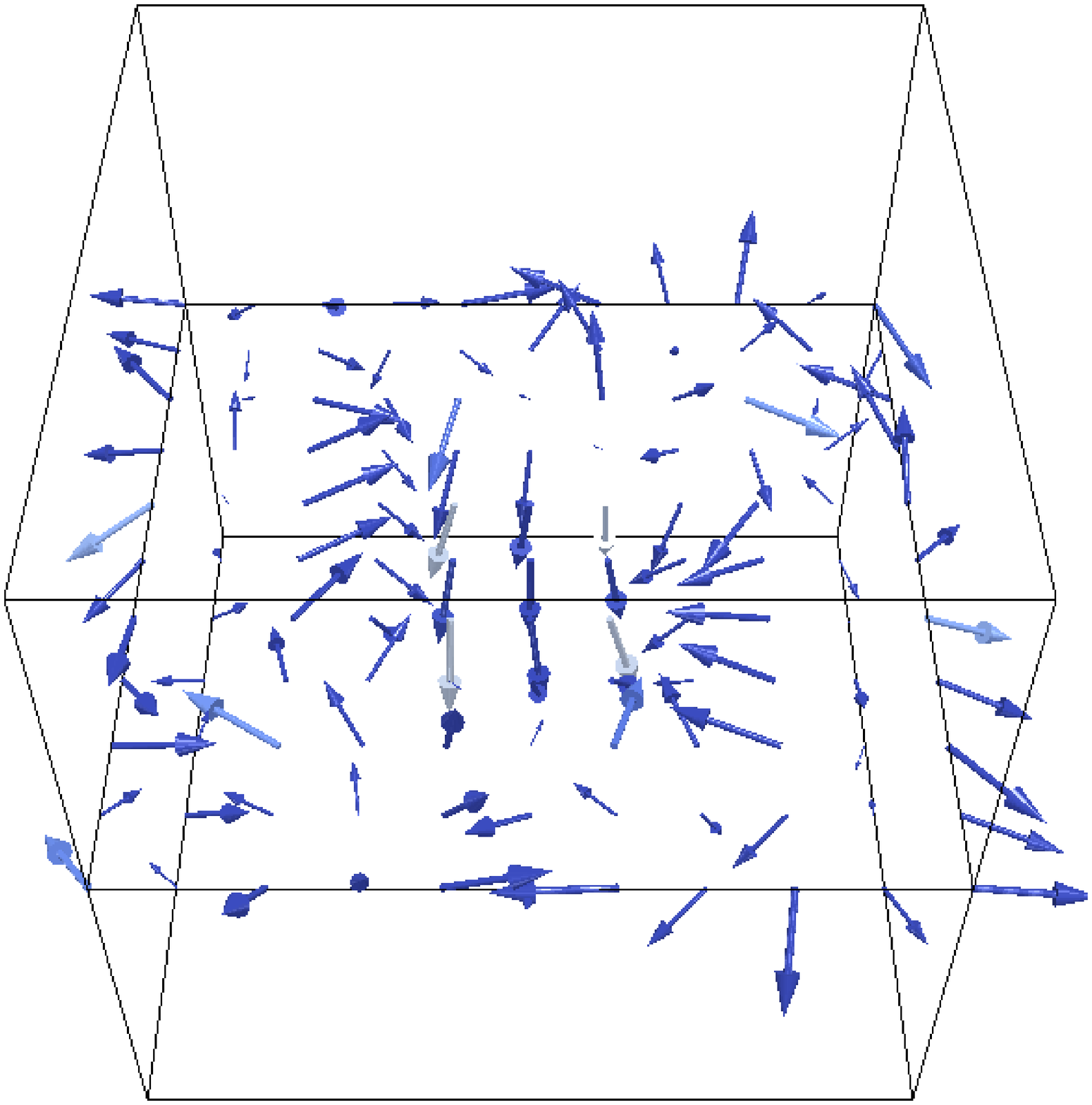}
\includegraphics[width=0.24\textwidth]{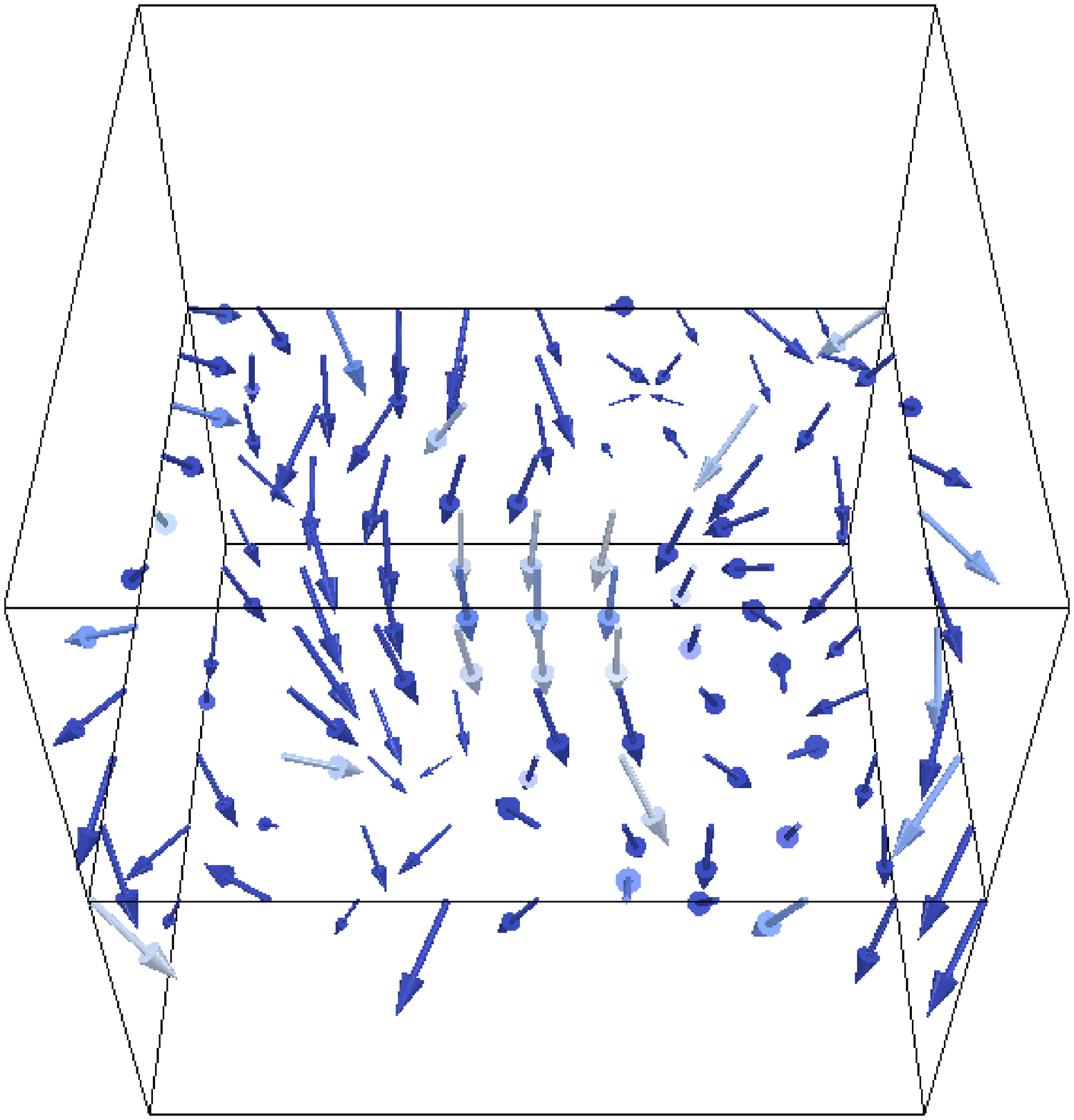}
\includegraphics[width=0.24\textwidth]{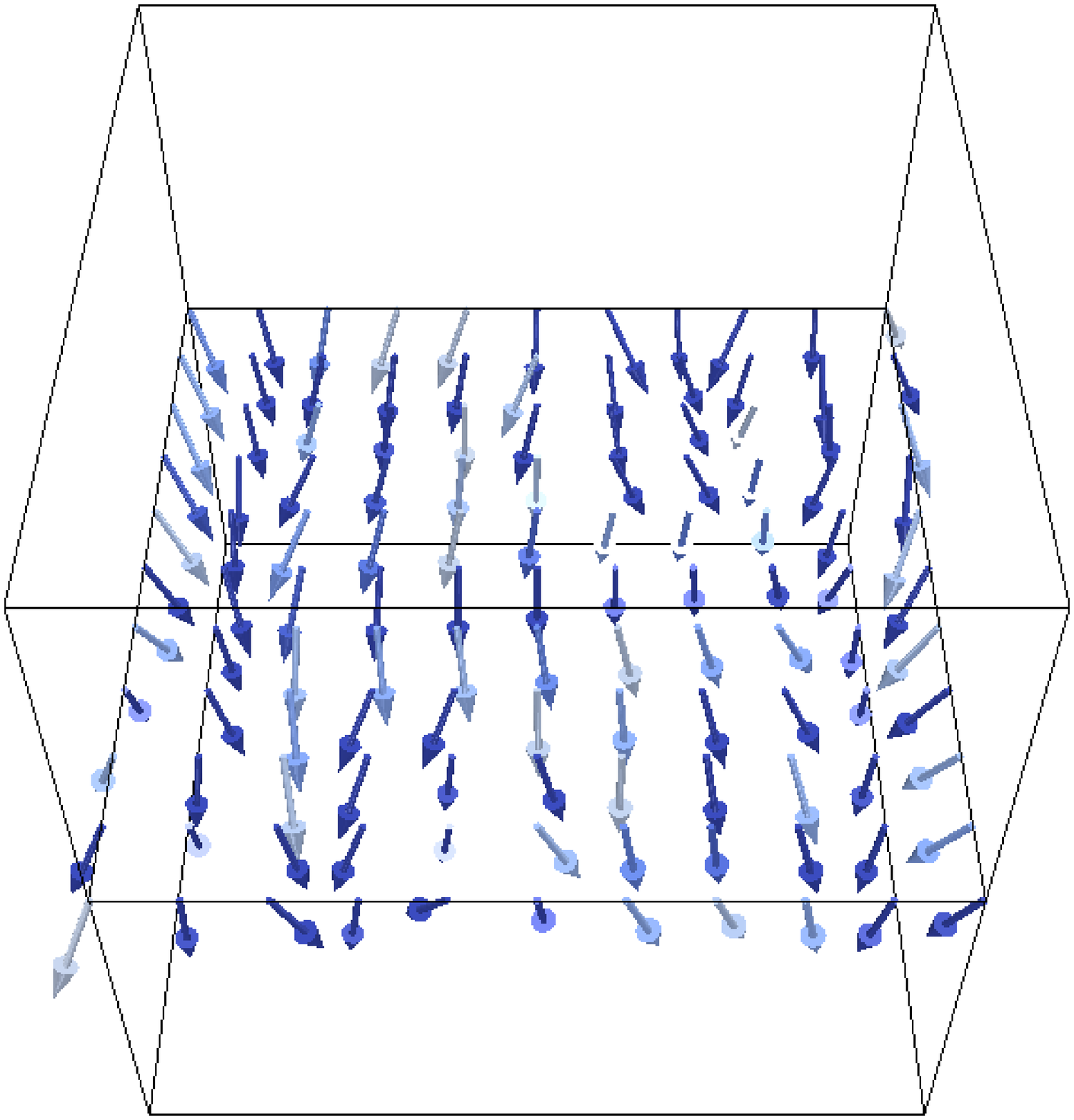}
\includegraphics[width=0.24\textwidth]{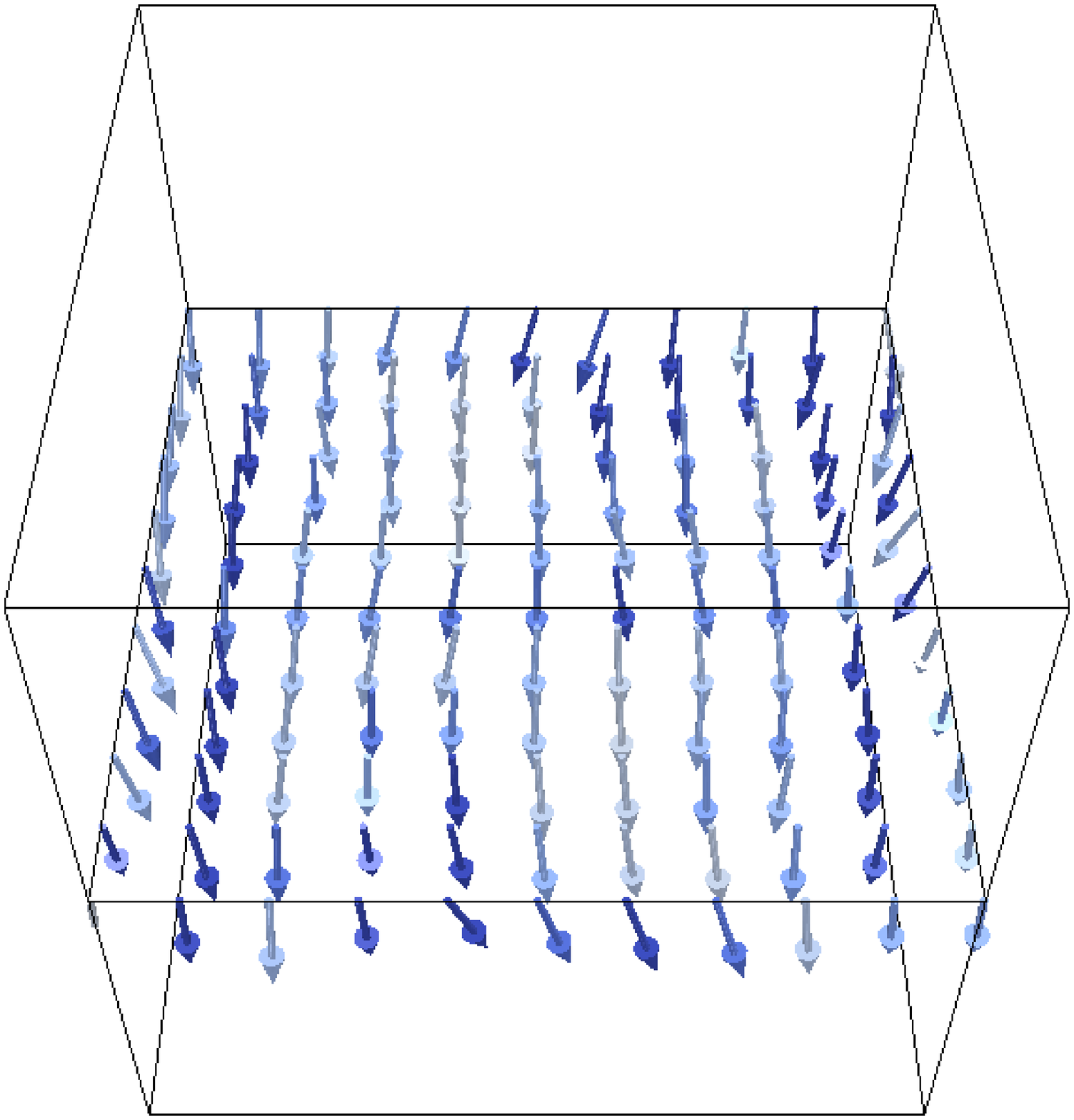}
\includegraphics[width=0.24\textwidth]{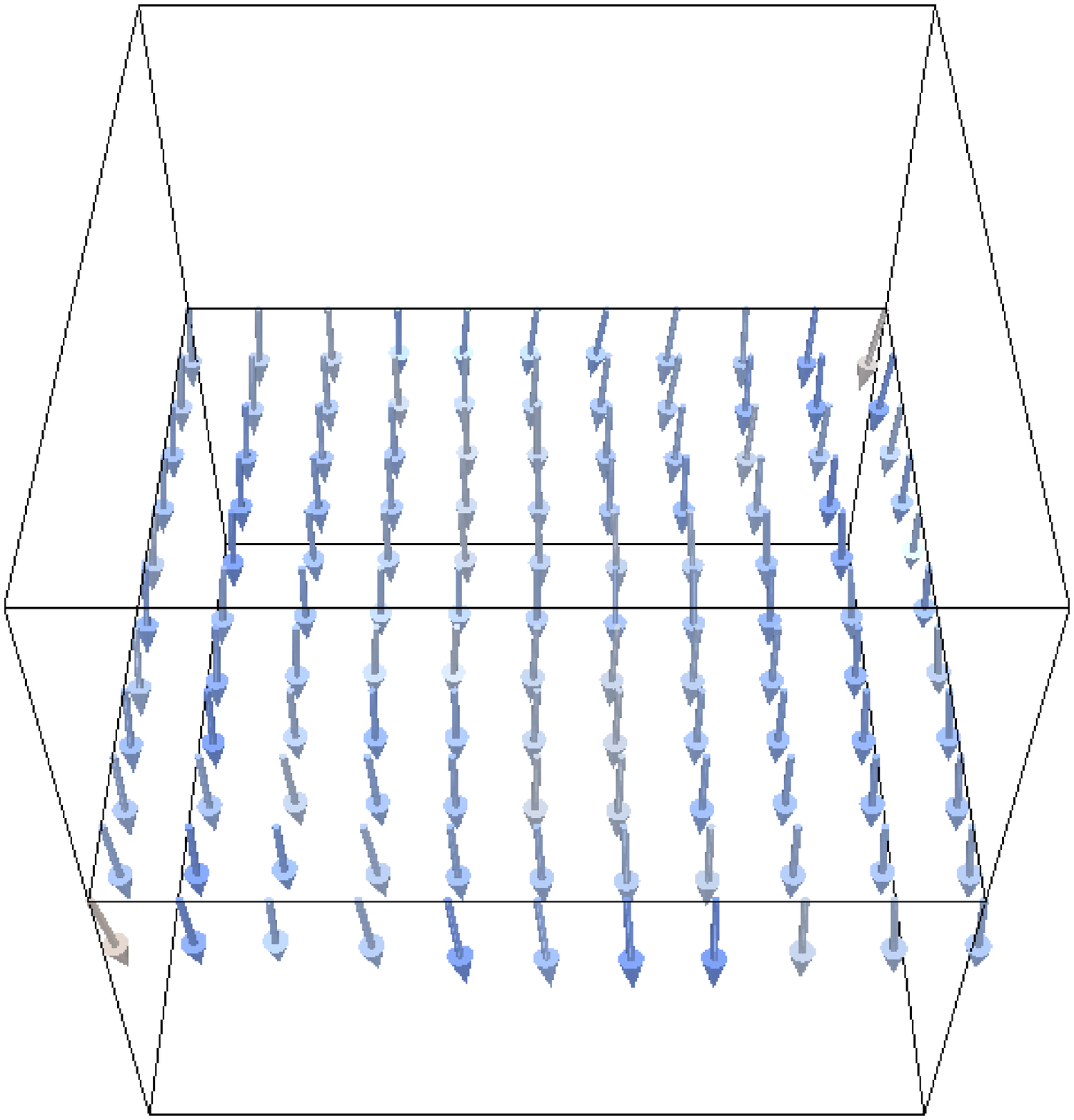}
\includegraphics[width=0.24\textwidth]{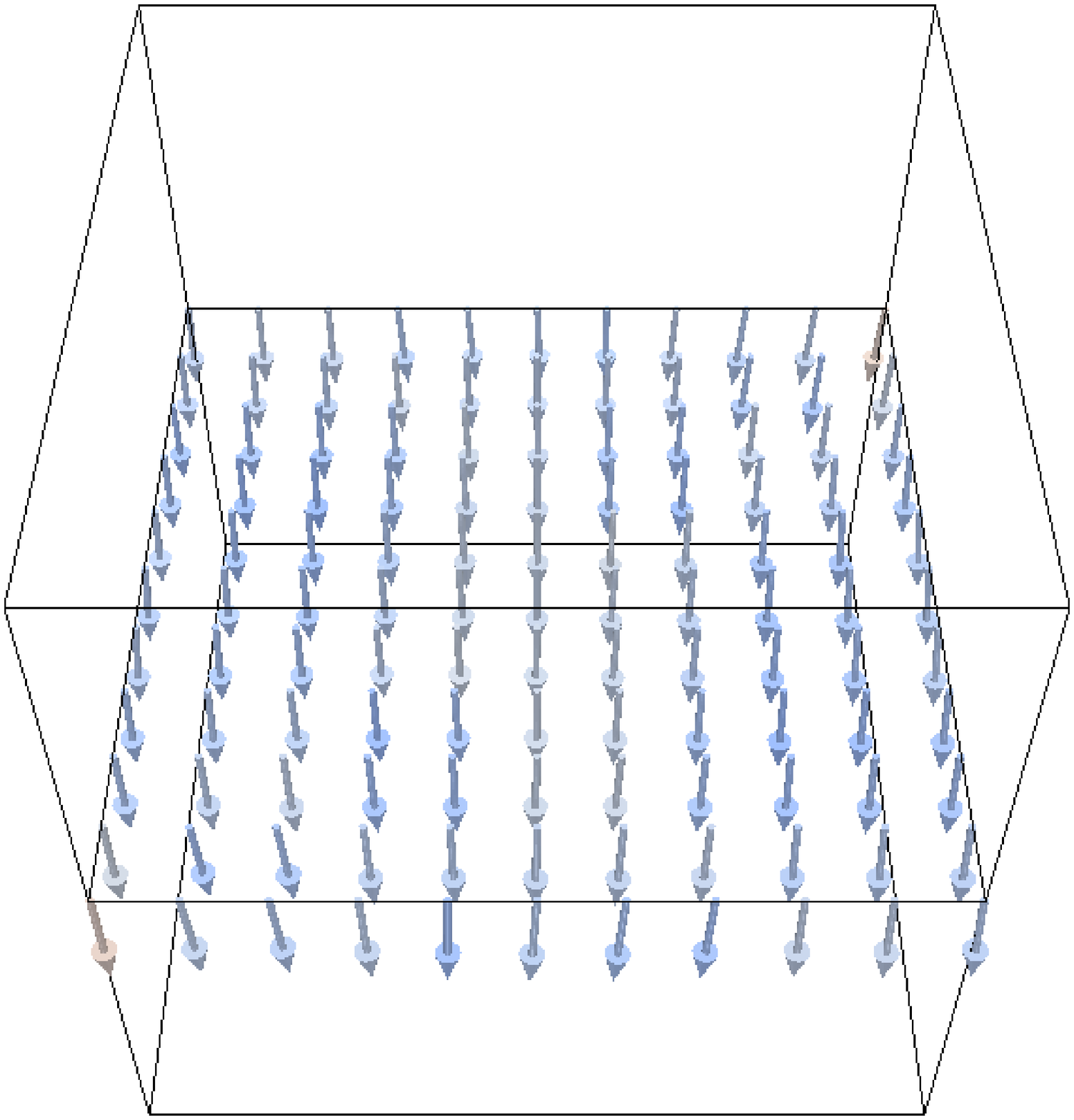}
\includegraphics[width=0.24\textwidth]{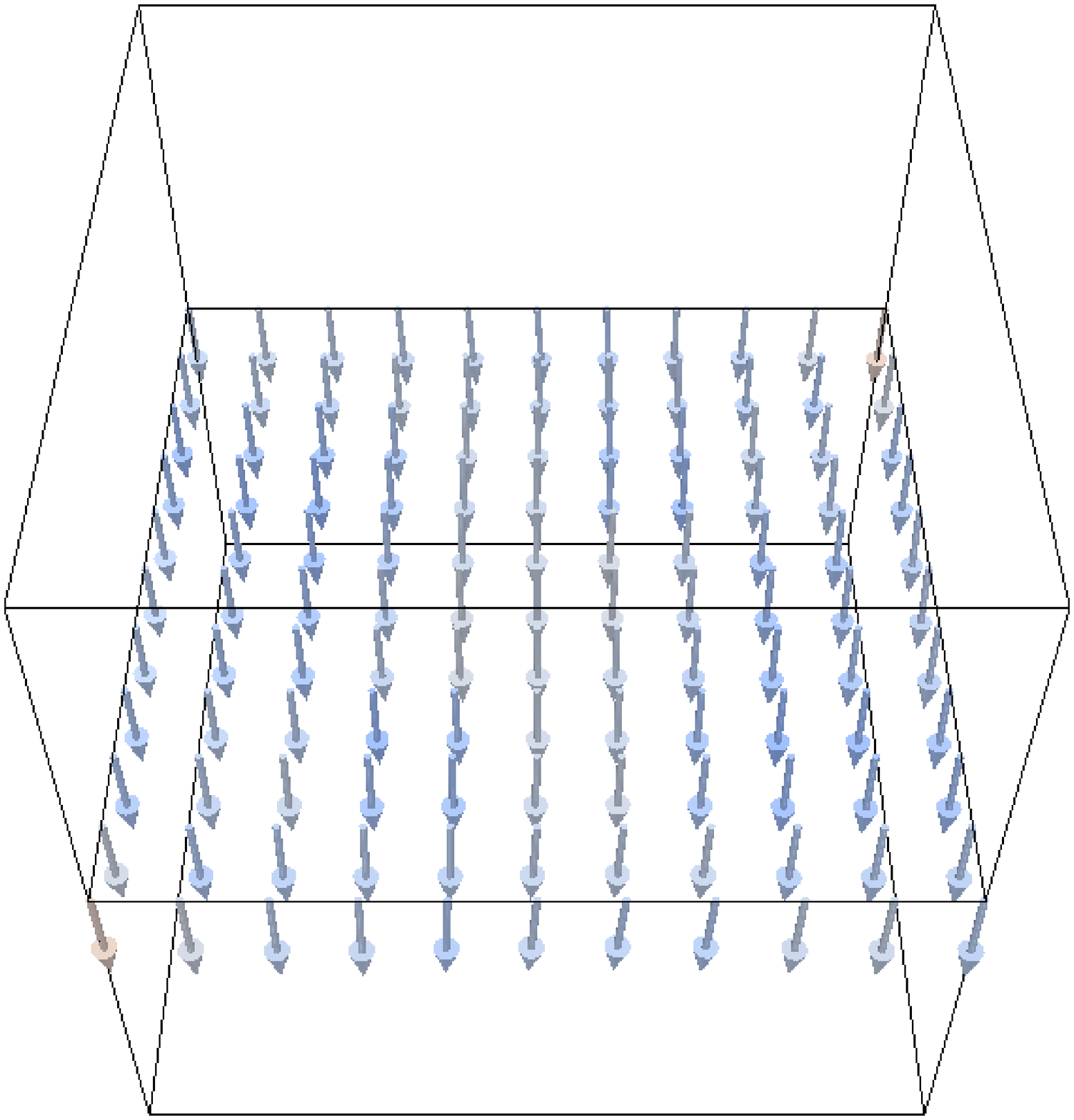}
\includegraphics[width=0.24\textwidth]{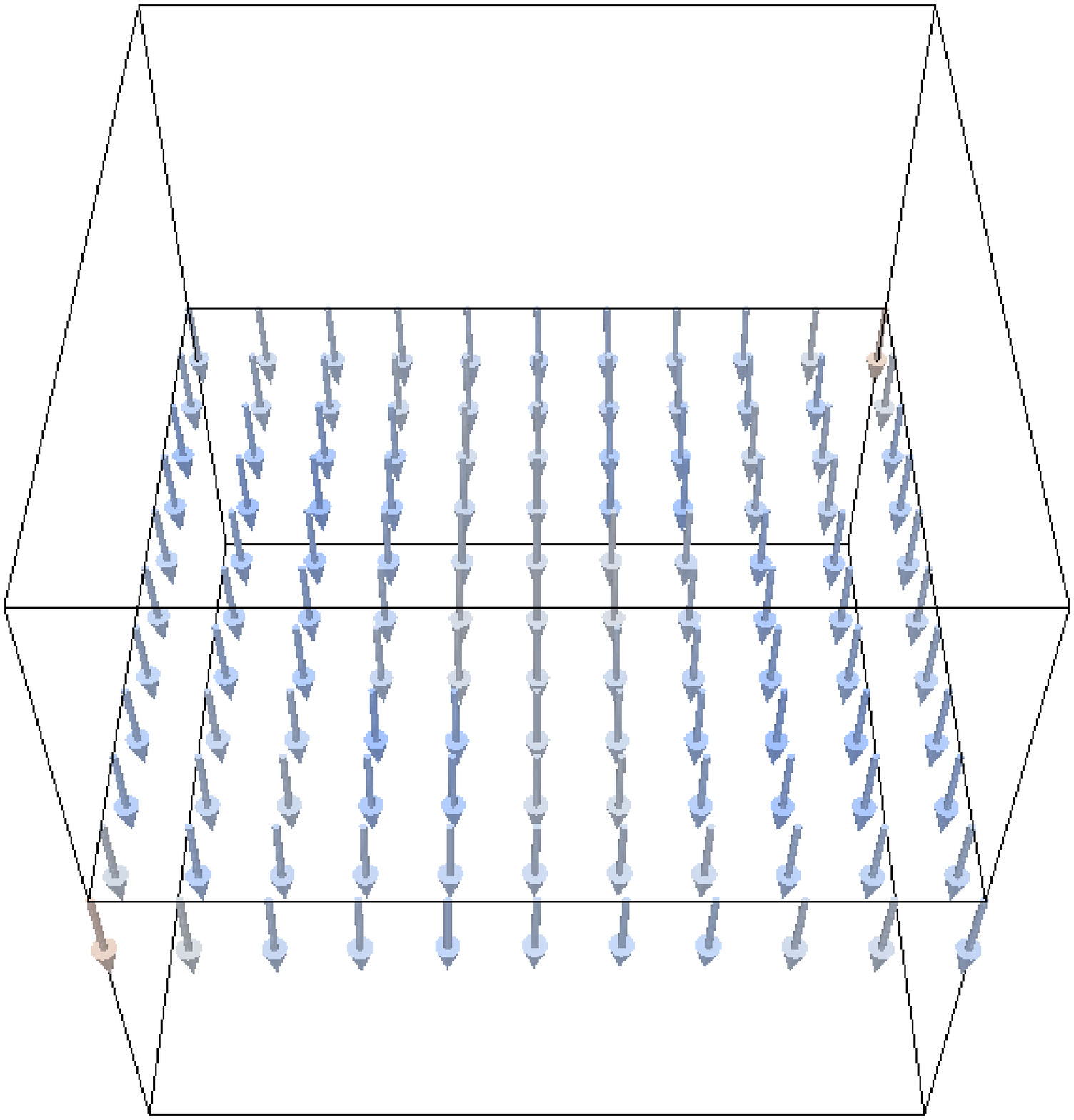}
\includegraphics[width=0.24\textwidth]{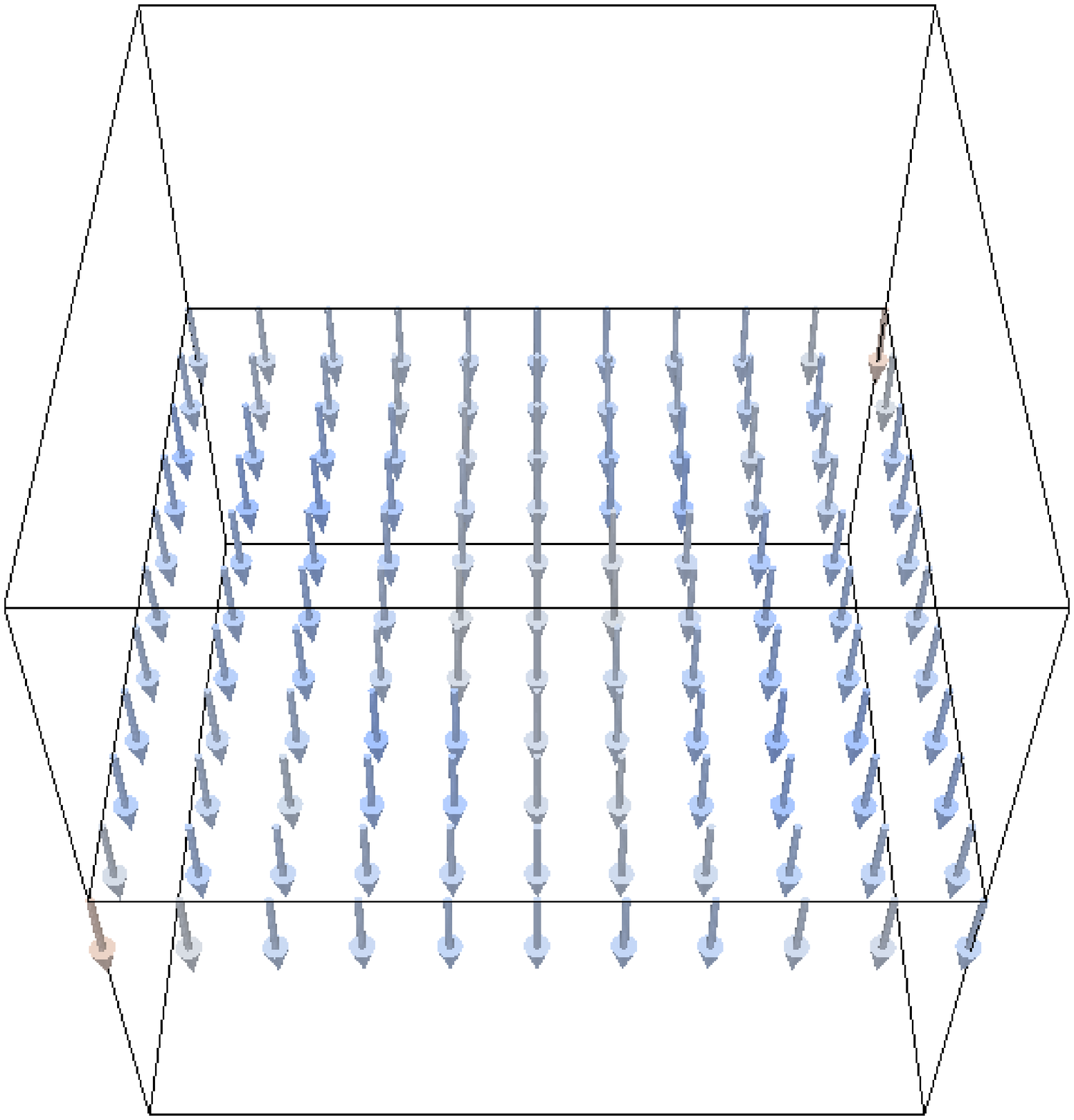}
\includegraphics[width=0.24\textwidth]{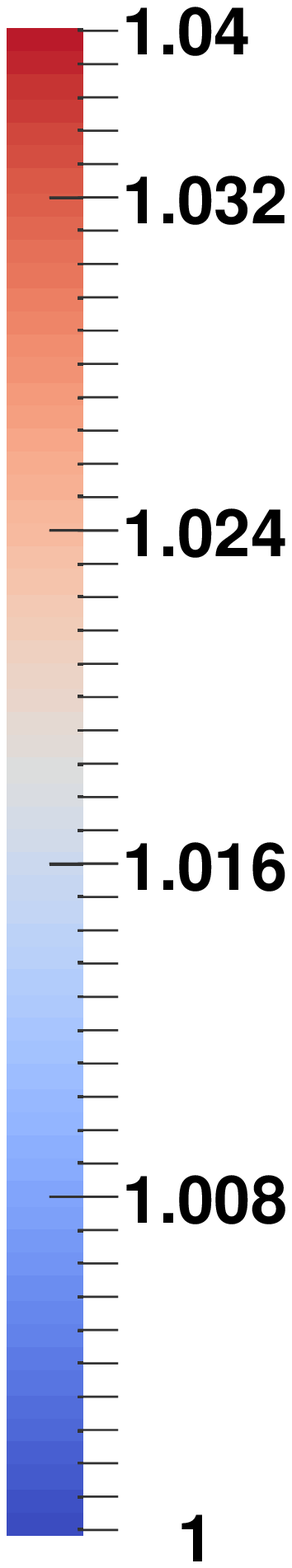}
\caption{Slice of the magnetization $\bm_{hk}(t_i)$ at $[0,1]^2\times \{1/2\}$ for $i=0,\ldots,10$ with $t_i=0.2i$. The color of the vectors represents the magnitude $|\bm_{hk}|$.
We observe that the magnetization aligns itself with the initial magnetic field $\bH^0$ by performing a damped precession.}
\label{fig:m}
\end{figure}

\begin{figure}
\includegraphics[width=0.24\textwidth]{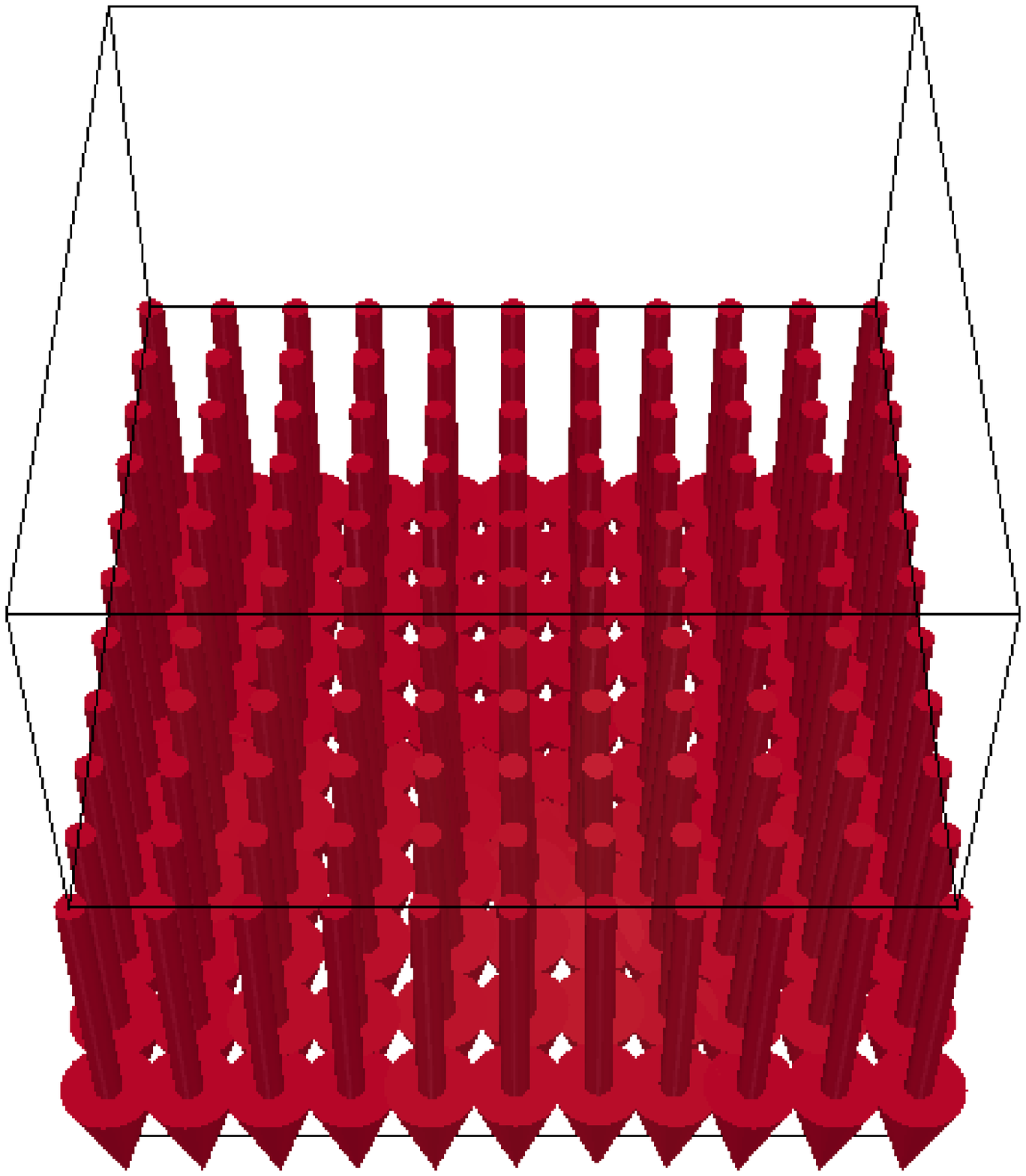}
\includegraphics[width=0.24\textwidth]{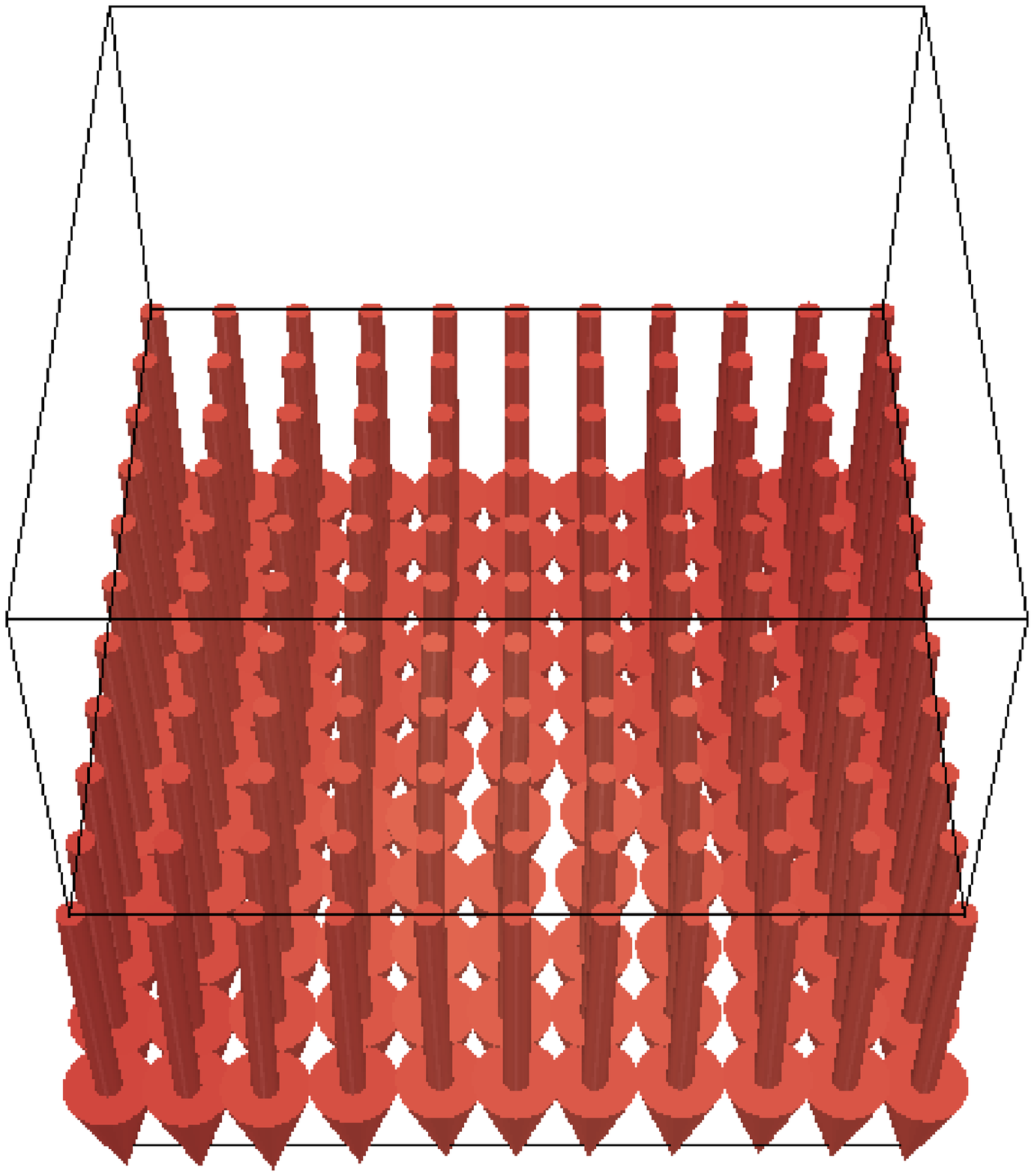}
\includegraphics[width=0.24\textwidth]{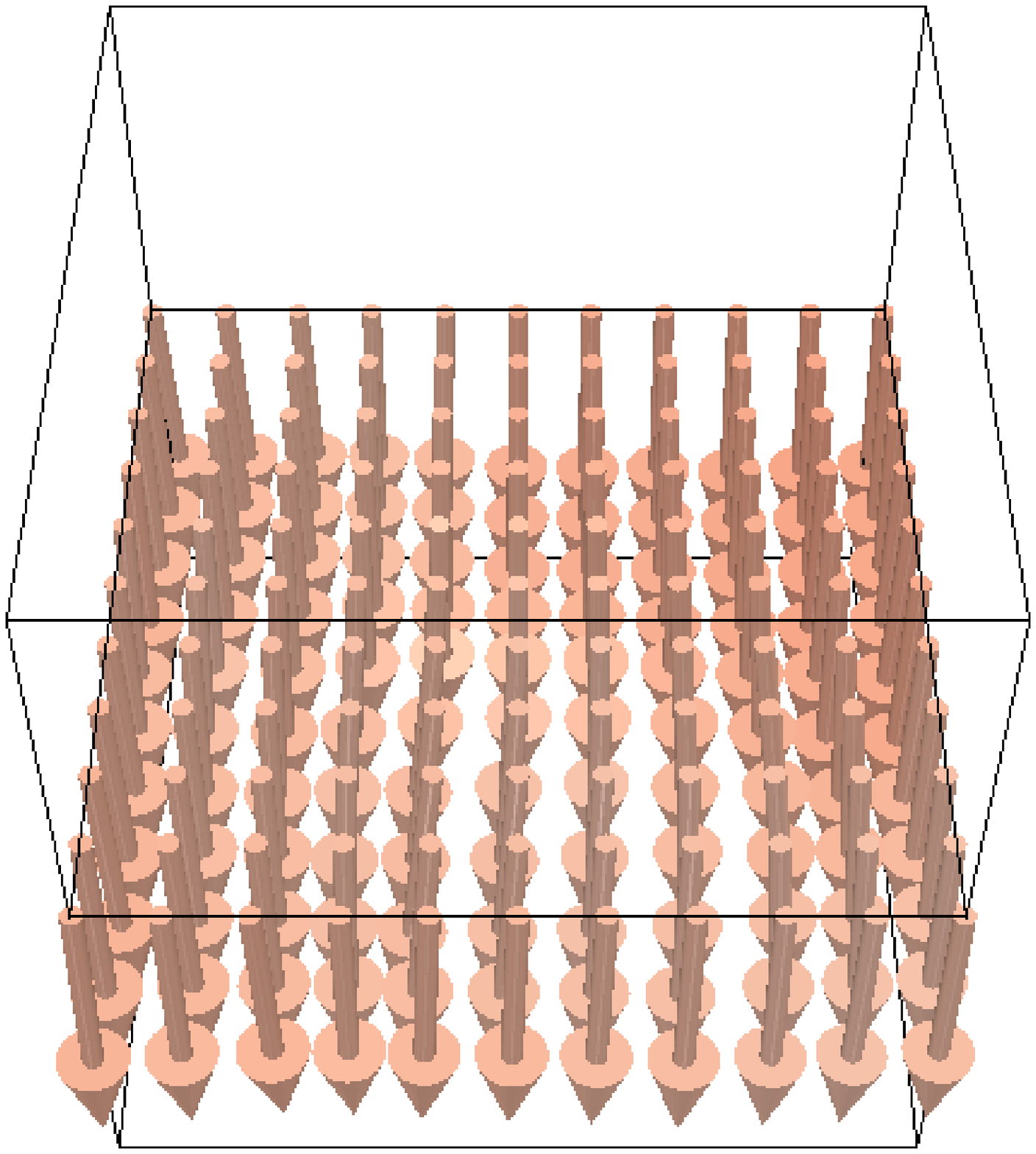}
\includegraphics[width=0.24\textwidth]{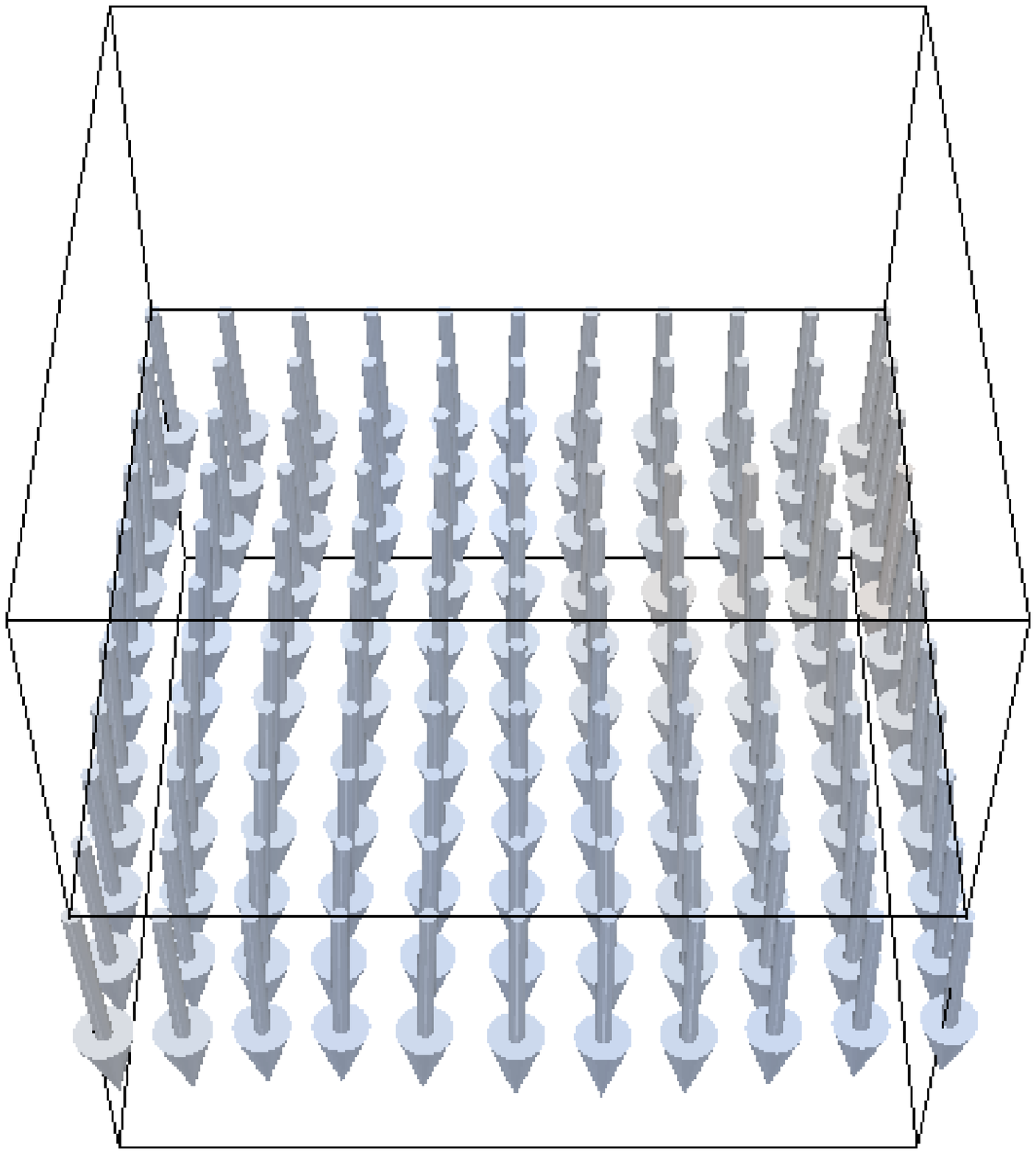}
\includegraphics[width=0.24\textwidth]{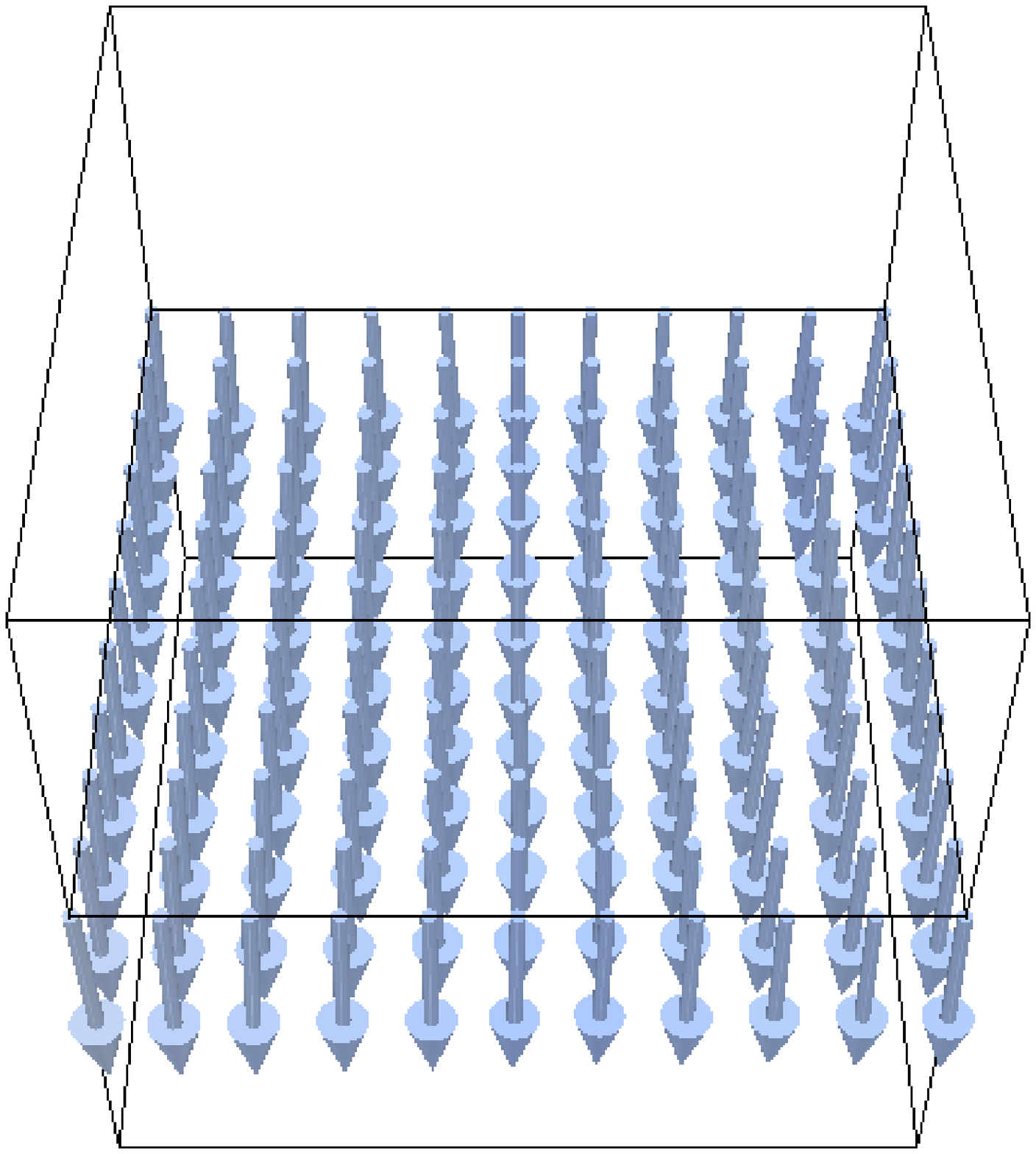}
\includegraphics[width=0.24\textwidth]{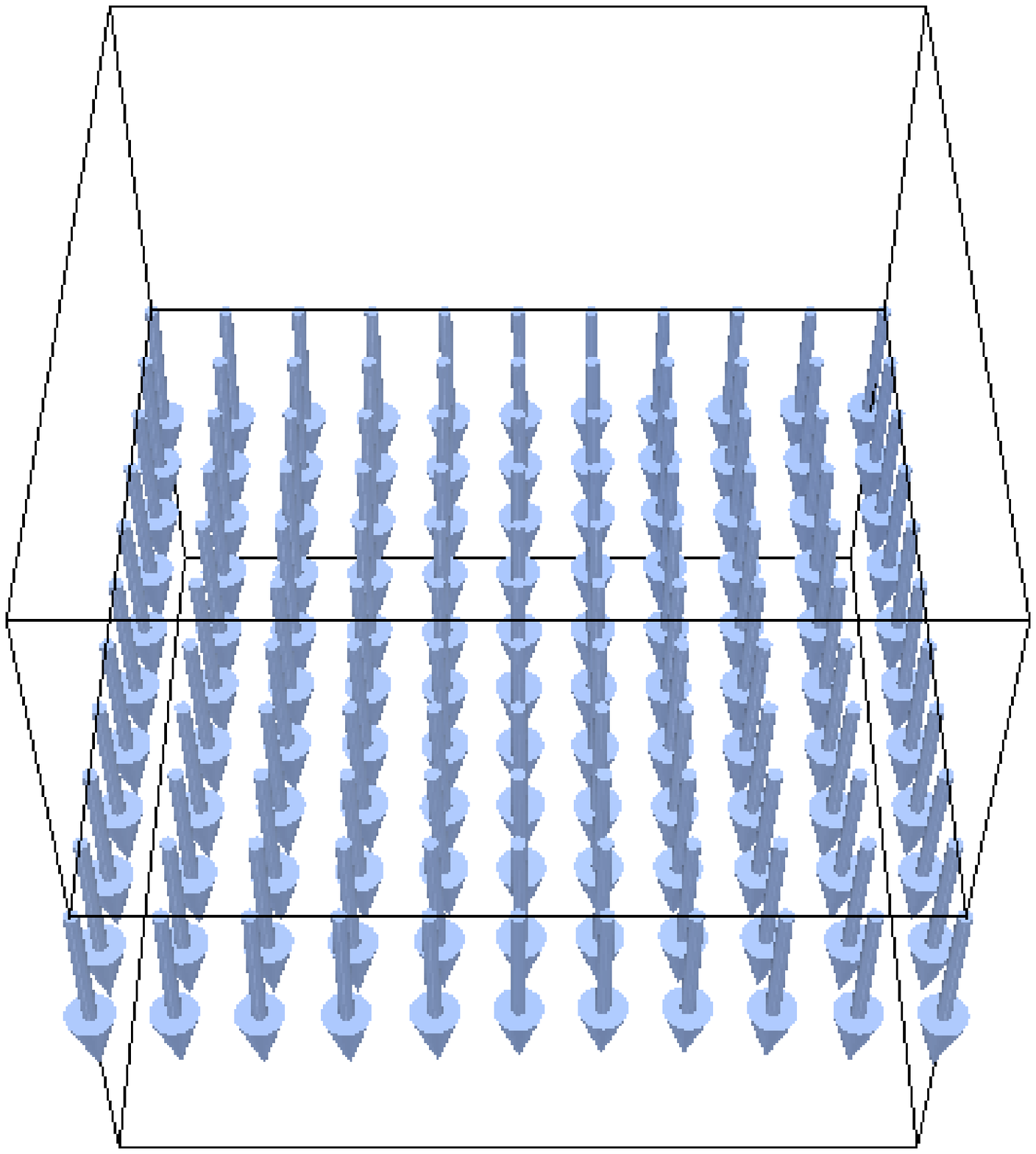}
\includegraphics[width=0.24\textwidth]{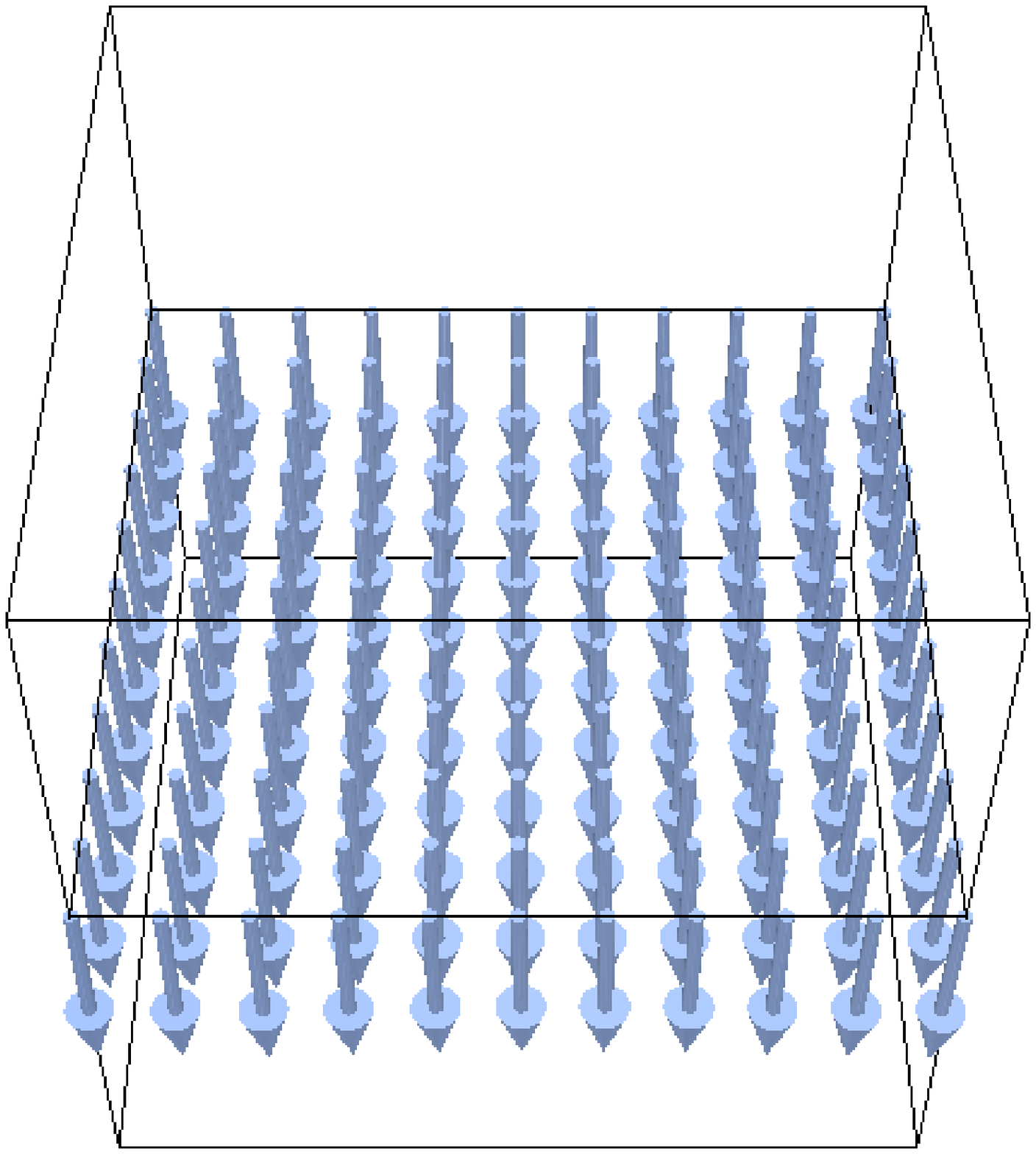}
\includegraphics[width=0.24\textwidth]{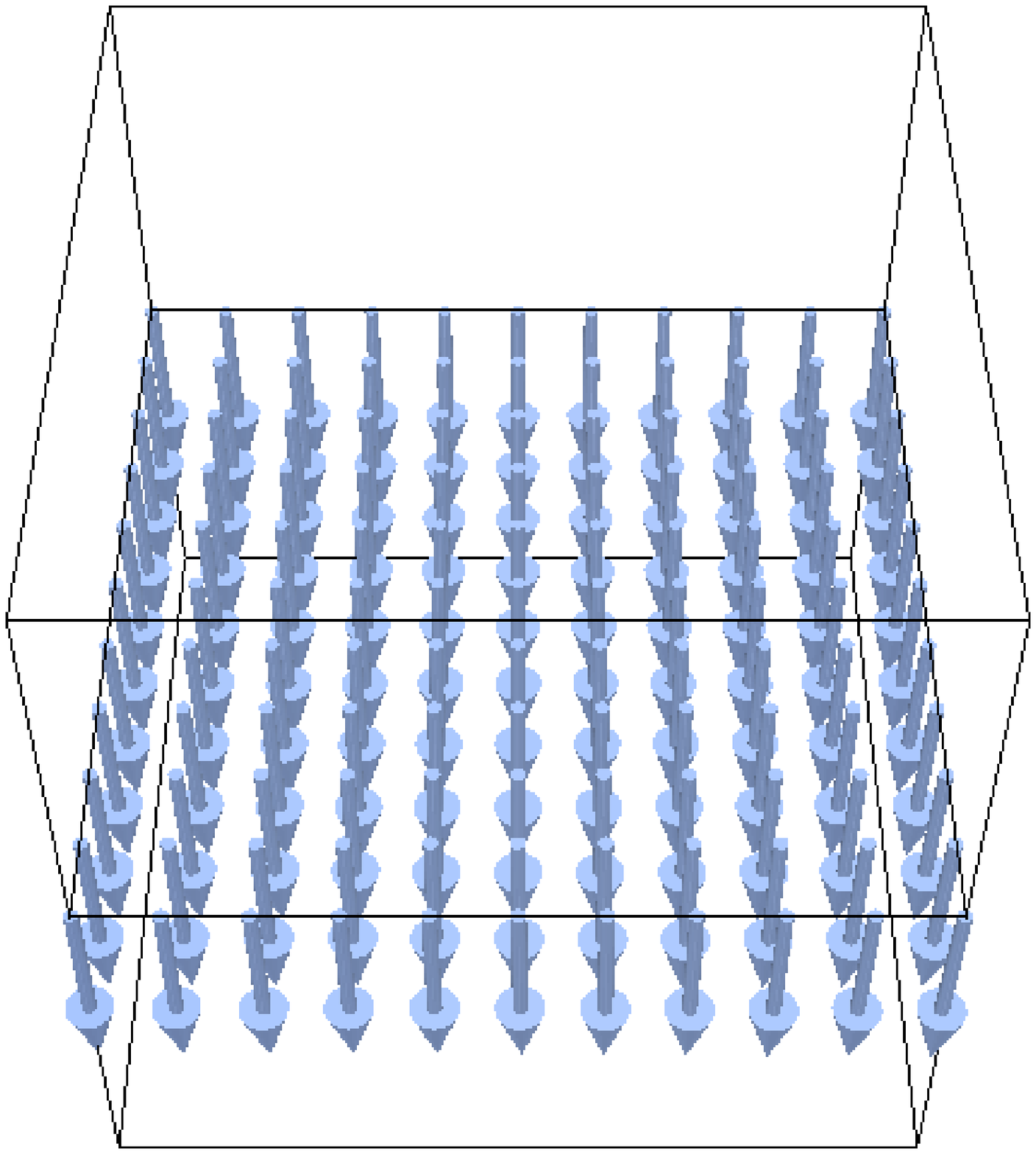}
\includegraphics[width=0.24\textwidth]{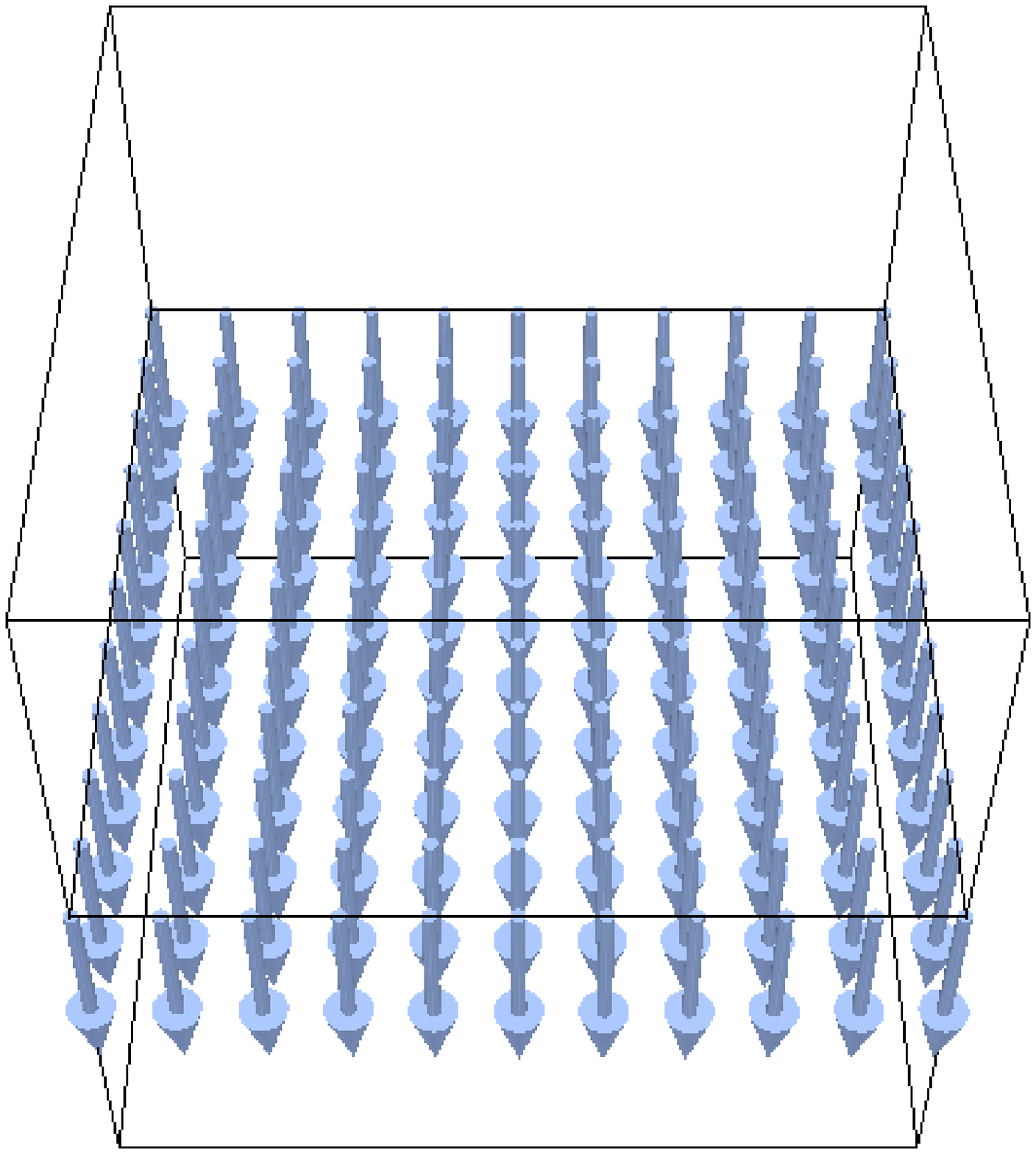}
\includegraphics[width=0.24\textwidth]{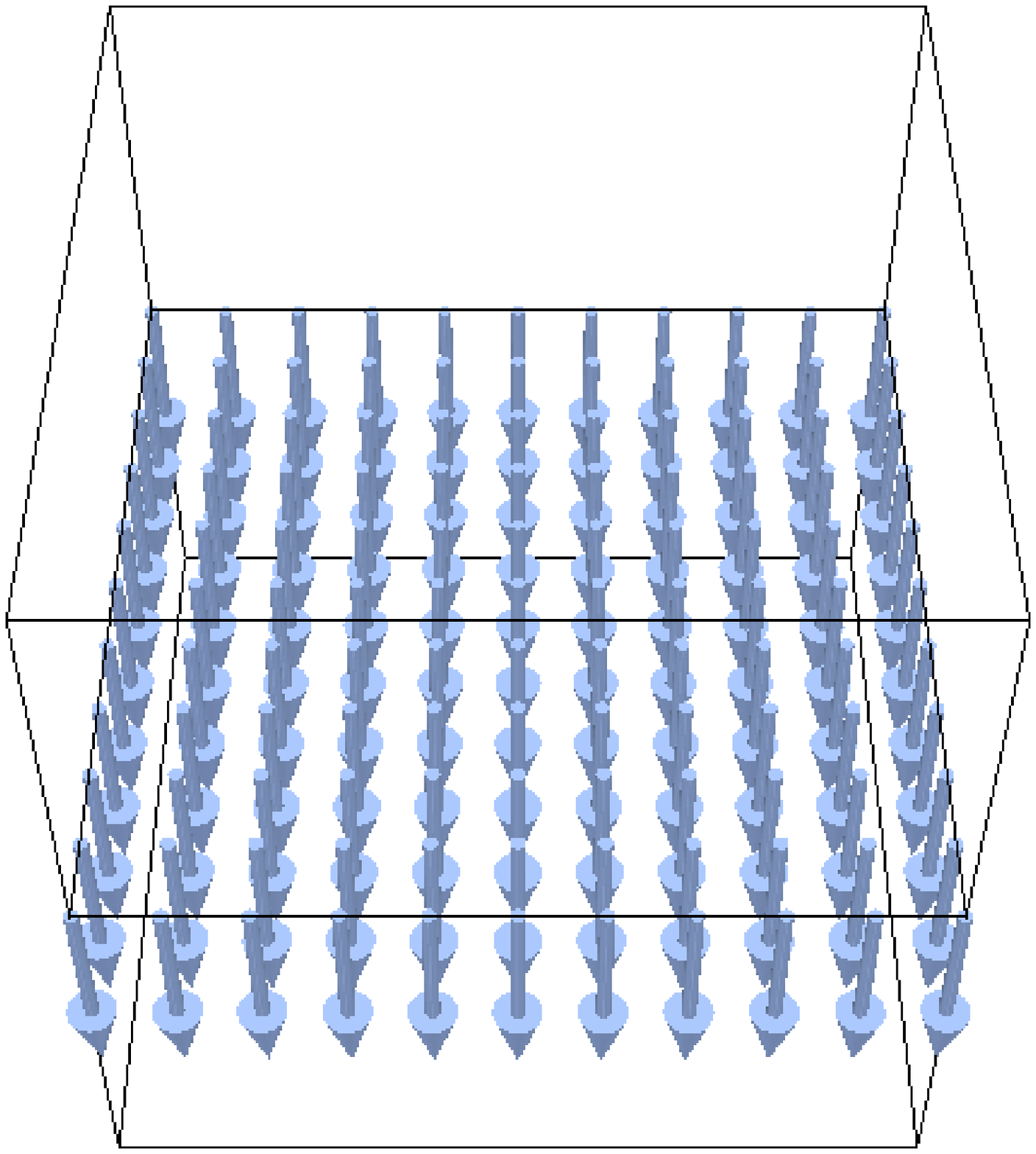}
\includegraphics[width=0.24\textwidth]{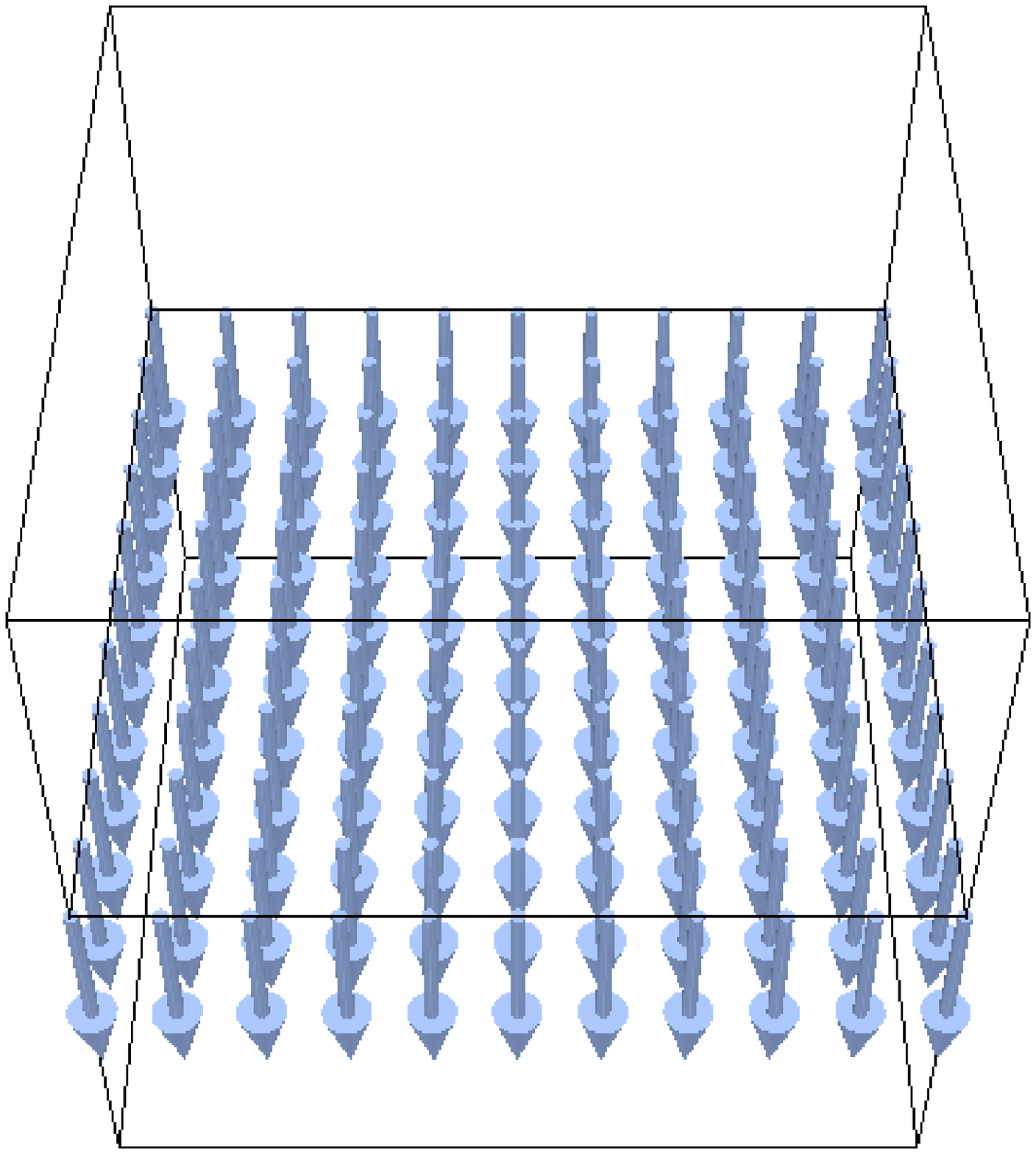}
\includegraphics[width=0.24\textwidth]{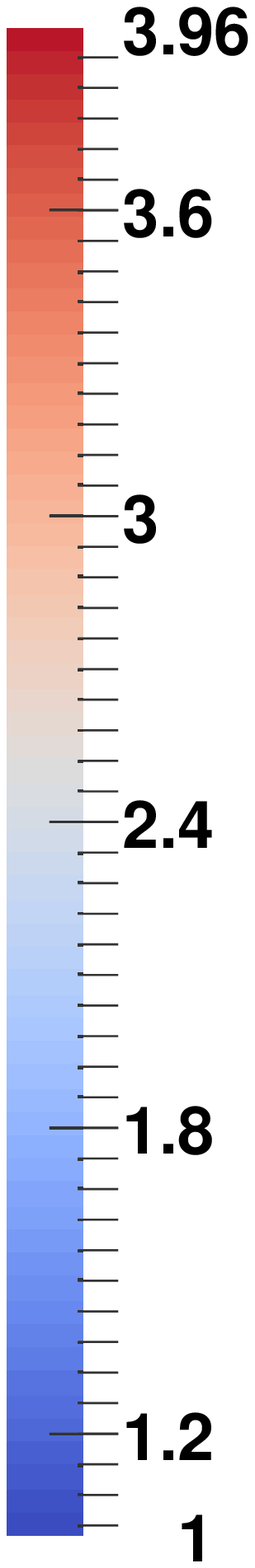}
\caption{Slice of the magnetic field $\bH_{hk}(t_i)$ at $[0,1]^2\times \{1/2\}$ for $i=0,\ldots,10$ with $t_i=0.2i$. The color of the vectors represents the magnitude $|\bH_{hk}|$.
We observe only a slight movement in the middle of the cube combined with an overall reduction of field strength.}
\label{fig:h}
\end{figure}

\bibliographystyle{myabbrv}
\bibliography{literature}
\end{document}